\documentclass[12pt,reqno]{amsart}

\usepackage[T1]{fontenc}
\usepackage{enumitem}
\usepackage{hyperref}
\hypersetup{
  colorlinks   = true,
  citecolor    = blue,
  linkcolor    = blue
}

\usepackage{amsmath,amsfonts,amsthm,color,tikz, comment, xcolor, xfrac,amssymb}
\usepackage[normalem]{ulem}
\usepackage{mathrsfs}
\usepackage{pdfsync}
\usepackage[font={scriptsize}]{caption}
\usepackage{environ}

\usepackage[numbers]{natbib}

\usepackage[left=1in, right=1in, top=1.1in,bottom=1.1in]{geometry}
\numberwithin{equation}{section}  % Numbers equations within sections

\newcommand{\der}{\delta}

\newcommand{\hr}[1]{{\textcolor{red}{#1}}}

\newcommand{\iot}{\int_{0}^{t}}

\newcommand{\pt}{\partial}

\newcommand{\bzeta}{{\zeta}}

\newcommand{\R}{\mathbb R}

\newcommand{\bp}{\mathbf{P}}

\newcommand{\cac}{\mathcal C}

\newcommand{\cf}{\mathcal F}
\newcommand{\cg}{\mathcal G}

\newcommand{\ci}{\mathcal I}
\newcommand{\cj}{\mathcal J}
\newcommand{\ck}{\mathcal K}

\newcommand{\cn}{\mathcal N}
\newcommand{\cp}{\mathcal P}
\newcommand{\cq}{\mathcal Q}
\newcommand{\cs}{\mathcal S}

\newcommand{\cv}{\mathcal V}

\newcommand{\al}{\alpha}

\newcommand{\ep}{\varepsilon}

\newcommand{\ga}{\gamma}
\newcommand{\gga}{\Gamma}
\newcommand{\ka}{\kappa}
\newcommand{\la}{\lambda}
\newcommand{\laa}{\Lambda}

\newcommand{\si}{\sigma}

\newcommand{\lp}{\left(}
\newcommand{\rp}{\right)}
\newcommand{\lc}{\left[}
\newcommand{\rc}{\right]}
\newcommand{\lcl}{\left\{}
\newcommand{\rcl}{\right\}}
\newcommand{\lln}{\left|}
\newcommand{\rrn}{\right|}

\newtheorem{theorem}{Theorem}[section]

\newtheorem{corollary}[theorem]{Corollary}

\newtheorem{definition}[theorem]{Definition}

\newtheorem{hypothesis}[theorem]{Hypothesis}
\newtheorem{lemma}[theorem]{Lemma}
\newtheorem{notation}[theorem]{Notation}

\newtheorem{proposition}[theorem]{Proposition}

\theoremstyle{remark}
\newtheorem{remark}[theorem]{Remark}

\theoremstyle{remark}

\newcommand{\bean}{\begin{eqnarray*}}
\newcommand{\eean}{\end{eqnarray*}}
\newcommand{\ben}{\begin{enumerate}}
\newcommand{\een}{\end{enumerate}}
\newcommand{\beq}{\begin{equation}}
\newcommand{\eeq}{\end{equation}}

\begin{document}

\title[Maximum Principle and Q-functions]
%{The Maximum Principle and Q-functions for\\ Relaxed Controls in a Rough Environment}
{The Pontryagin Maximum Principle and Q-functions in Rough Environments}

\author[E. Ashkarian]{Estepan Ashkarian}
\address{E. Ashkarian: Department of Mathematics, 
Purdue University, West Lafayette}
\email{eashkari@purdue.edu}

\author[P. Chakraborty]{Prakash Chakraborty}
\address{P. Chakraborty: Marcus Department of Industrial and Manufacturing Engineering,
The Pennsylvania State University, State College}
\email{prakashc@psu.edu}

\author[H. Honnappa]{Harsha Honnappa}
\address{H. Honnappa: Edwardson School of Industrial Engineering, Purdue University, West Lafayette}
\email{honnappa@purdue.edu}

\author[S. Tindel]{Samy Tindel}
\address{S. Tindel: Department of Mathematics, 
Purdue University, West Lafayette}
\email{stindel@purdue.edu}

\maketitle

\begin{abstract}
We derive the Pontryagin maximum principle and $Q$-functions for the relaxed control of noisy rough differential equations. Our main tool is the development of a novel differentiation procedure along `spike variation' perturbations of the optimal state-control pair. We then exploit our development of the infinitesimal $Q$-function (also known as the $q$-function) to derive a policy improvement algorithm for settings with entropic cost constraints.
\end{abstract}

\section{Introduction}

The current contribution has to be seen as a step forward in a long term project, aiming at modeling reinforcement learning in a general noisy and non Markovian environment. The model we have considered so far can be summarized as follows: 
Consider a path $(x^\gamma_t : t \geq 0)$ which represents the state of a controlled process.  Let also $(\gamma_t : t \geq 0)$ be a curve of relaxed controls; 
that is, a path taking values in the set $\cp(U)$ of probability measures on a given control space $U$.
We are given a general noisy environment, which is modeled by a rough path $(\zeta_{t} : t \geq 0)$.
The state process $x^{\ga}$ is then the solution of the following rough differential equation,
\begin{equation}\label{d1}
x_t^{\ga} = x_0 + \iot \int_{U} b(s, x_s^{\ga}, a) \, \ga_s(da) ds + \iot \si(s, x_s^{\ga}) \, d\zeta_s \, ,
\end{equation}
where the last integral in~\eqref{d1} has to be interpreted in the rough sense.  More generally we are interested in rough differential equations of the form
\beq\label{eq:rde-1}
x_t^{\ga} = x_0 + \iot  B(s, x_s^{\ga}, \ga_s) \, ds + \iot \si(s, x_s^{\ga}) \, d\zeta_s,
\eeq
where the coefficients $B$ and $\si$ enjoy convenient continuity properties with respect to their arguments. Observe that the coefficient $\si$ in~\eqref{eq:rde-1} does not depend on the control $\ga$, in order to avoid degeneracy problems of the kind mentioned in \cite{diehl,allan-cohen}.
With this framework in hand,
the control problem of interest in this paper consists in maximizing an objective functional over $[0,T]$, for a given time horizon $T>0$. This functional takes the form
\beq\label{eq:cont-1}
	J_T(\gamma) \equiv \int_0^T F(s,x_s^\gamma,\gamma_s) ds + G(x_T^\gamma),
	\quad\text{for}\quad \ga\in\mathcal{K} \, ,
\eeq
where $\mathcal K$ is an appropriate class of \emph{admissible} relaxed control curves. 

\begin{remark}
The setting summarized above might seem purely deterministic, with a rough path $\zeta$ satisfying standard assumptions. However, one of the main applications of the general framework is in the realm Gaussian processes. In particular, fractional Brownian motions are covered by our results. See Remark~\ref{rmk:fbm} for more details about this claim.
\end{remark}

Building on our prior work~\cite{CHT}, in Theorem~\ref{Pontryagin} we derive a Pontryagin maximum principle and establish the neccesary condition satisfied by the optimal state-(relaxed) control pair in the rough context. Our method of proof is based on a novel differentiation procedure along perturbations, termed `spike variations', of the optimal couple. Our definition of the perturbations allow us to circumvent the fact that the policy space $\mathcal{K} \subset \mathcal P(U)$ is not a vector space, thereby preventing the use of standard functional derivatives. Our proof of Theorem~\ref{Pontryagin} is novel and takes a completely different approach to~\cite[Theorem 7]{diehl} in the non-relaxed case. The latter uses rough flow transformations to reduce the noisy problem to a deterministic problem with random coefficients. Our approach exploits the differentiability of the solutions of the RDE with respect to the perturbation, and uses known estimates of the Jacobian of the flow map. We refer to~\cite{Lew} for a related notion of needle-like variation.

The close relationship between the Hamiltonian in optimal control and the Q-function in (continuous time) reinforcement learning has been explored in~\cite{MehtaMeyn2009}. Specifically, the reference~\cite{MehtaMeyn2009} studies a stationary reward setting wherein the Hamiltonian is given by $H(x,u,p) = c(x,u) + p^T f(x,u)$. The Pontryagin maximum principle implies that the optimal co-state $p$  is equal to $\nabla J^*(x)$, where $J^*$ is the optimal value function, and therefore we have the Hamilton-Jacobi-Bellman functional $H^*(x,u) = H(x,u,p^*)$. The article~\cite{MehtaMeyn2009} points out that $H^*(x,u)$ is precisely the Q-function in continuous time, since the Pontryagin maximum principle implies that the optimal control can be found by solving $\max_u H^*(x,u)$, which is precisely analogous to the Q-function in reinforcement learning. Consequently, Q-learning can be viewed as a data-driven way to learn the Hamiltonian in Pontryagin's principle. 

Here, we extend this understanding of the Q-function to the noisy, relaxed control setting in this paper, by exploiting the `spike variation' perturbation developed in our derivation of the Pontryagin maximum principle. Unlike the setting in~\cite{MehtaMeyn2009}, where the controls were fully deterministic, the appropriate notion to consider with randomized controls is the $q$-function. Note that the $q$-function had been defined in~\cite{jia2023qlearning} as the infinitesimal version of the Q-function, taking limits on spike variation perturbations.~\cite{jia2023qlearning} derive the $q$-function when the state dynamics are a semi-martingale (specifically driven by Brownian motion), exploiting Ito's formula. This does not work in our setting, of course, and we derive the $q$-function in Proposition~\ref{prop:qfunc} by exploiting the novel small interval perturbation technique we developed for the Pontryagin's maximum principle.

We further explore cases of reward functions explicitly containing entropy terms, enabling us to translate the $q$-function into a computational tool for policy improvement in continuous time. In Section~\ref{sec:open-loop}, we consider open loop controls and derive an expression for the optimal policy, specifically showing that under entropic considerations, the optimal policy is of a Gibbs-form. This parallels the existing literature, however these Gibbs-form policies are implicitly (noise) path dependent through the gradient of the value function. Furthermore, the optimal open loop policy is shown to be explicitly related to the $q$-function in Proposition~\ref{prop:closed-gibbs-q}. We further develop the computational angle in Section~\ref{sec:closed-loop}, by considering closed loop policies, corresponding to the situation where one has the ability to improve the current policy according to the data received. While standard policy improvement methods explicitly use the monotonicity of the HJB equation, this does not immediately hold in our noisy setting. To address this, we instead use a composition with rough flows and transform the noisy HJB equation into an ordinary one with noisy coefficients. This allows us to use a comparison principle, arising from the interpretation of the ordinary PDE in the viscosity sense, to demonstrate the policy improvement for the closed loop policy.

As the reader might see, this article provides tools for the numerical analysis of relaxed control problems in a general rough environment. For sake of conciseness, we have set aside other important questions. Among those let us mention: 
\textbf{(i)} Expanding on our Section~\ref{sec:gibbs}, one would like to get algorithmic versions of policy improvement methods together with convergence results. 
\textbf{(ii)}  As mentioned above, our Pontryagin principle and our definition of $q$-function depend on the quantity $\nabla J^*(x)$. As we will see in Section~\ref{sec: Pontry}, this quantity anticipates on the future of the observations. This is arguably the main pitfall of a pathwise approach to control in a general noisy environment (though, this may be reasonable in an `offline' reinforcement learning setting, where full sample paths of the controlled processes are available). At the heart of the problem lies the fact that  we are not able to use conditional expectations (like in the Brownian setting) in order to be reduced to an adapted setting. One possibility worth exploring in order to circumvent this issue is to use signature expansions, which can then be approximated by their expected values. See~\cite{BBHRN,FS-al} for an account on this type of approach in more standard control problems, as well as~\cite{OAI-al} for path tracking.
\textbf{(iii)} Discuss the generalized Hopf-Lax approach, which can be used to approximate the value function (seen as the viscosity solution to a rough Hamilton-Jacobi equation) without discretizing the state space (see~\cite{KY} for this approach in the deterministic case).
We plan on tackling those issues in a subsequent contribution.

The rest of the paper is organized as follows. We introduce our rough viscosity setting rigorously (recalling our results from~\cite{CHT}) in Section~\ref{sec:rough}. We derive our main result on the Pontryagin maximum principle in~\ref{sec: Pontry}, followed by the derivation of the $Q$- and $q$-function's in~Section~\ref{sec:Q}. We derive the open loop Gibbs measure representation for the optimal policy and (finally) explore policy improvement in Section~\ref{sec:gibbs}.

\section{The rough viscosity setting}~\label{sec:rough}

In this section we review the rough viscosity {solution theory} and the main results obtained in~\cite{CHT}. We first start by recalling some basic rough paths notions using controlled processes, borrowed from~\cite{gubinelli}. Before introducing our first notions, let us label a general notation widely used throughout the paper.

\begin{notation}
In the sequel $k$-th order simplexes on $[0,T]$ will be denoted by $\cs_{k}$. Namely 
$\cs_k([0,T]) = \{ (t_1, \ldots, t_k) \in [0,T]^k; t_1 \leq \cdots \leq t_k \}$. The norm of a vector $f$ in a vector space $V$ is written as $\cn [f; V]$. All constants $c$ can vary from line to line. Function spaces with space-time regularity are spelled as $\cac^{a,b}$.
\end{notation}

\subsection{Increments}\label{sec:increm}
For a vector space $V$ and an integer $k \geq 1$, let $\cac_k(V)$ be the set of functions $g:\mathcal{S}_{k}([0,T]) \to V$ such that $g_{t_1 \cdots t_k} = 0$ whenever $t_i = t_{i+1}$ for some $i \leq {k-1}$. Such a function will be called a $(k-1)$-increment, and we set $\cac_{*}(V) = \cup_{k \geq 1} \cac_{k}(V)$. Then the operator $\der : \cac_{k}(V) \to \cac_{k+1}(V)$ is defined as follows
\begin{equation}\label{eq:def-delta}
{\der g}_{t_1 \cdots t_{k+1}} = \sum_{i = 1}^{k+1} (-1)^{k-i} g_{t_1\cdots \hat{t_i} \cdots t_{k+1}} \, ,
\end{equation} 
where $\hat{t_{i}}$ means that this particular argument is omitted.
In particular, note that for $f \in \cac_1(V)$ and $h \in \cac_2(V)$ we have
\beq\label{eq:der}
\der f_{st} = f_t - f_s, \quad \text{and} \quad \der h_{sut} = h_{st} - h_{su} -h_{ut}. 
\eeq
It is easily verified that $\der \der = 0$ when considered as an operator from $\cac_{k}(V)$ to $\cac_{k+2}(V)$.

\noindent
The size of these $k$-increments are measured by H\"older norms defined in the following manner: for $f \in \cac_{2}(V)$ and $\mu > 0$ let
\begin{equation}\label{eq:norm_1}
\|f\|_{\mu} = \sup_{(s,t) \in \mathcal{S}_2([0,T])} \dfrac{|f_{st}|}{{|t-s|}^{\mu}},
\quad\text{  and  }\quad
\cac_2^{\mu}(V) = \lbrace f \in \cac_2(V); {\|f\|}_{\mu} < \infty \rbrace \, .
\end{equation}
The usual H\"older space $\cac_1^{\mu}(V)$ will then be determined in the following way: for a continuous function $g \in \cac_{1}(V)$, we simply set
\begin{equation*}
{\|g\|}_{\mu} = {\|\der g\|}_{\mu} \, ,
\end{equation*}
and we will say that $g \in \cac_1^{\mu}(V)$ iff ${\|g\|}_{\mu}$ is finite. 
\begin{remark} 
	Notice that ${\|\cdot\|}_{\mu}$ is only a semi-norm on $\cac_1(V)$, but we will generally work on spaces for which the initial value of the function is fixed.
\end{remark}

We shall also need to measure the regularity of increments in $\cac_3(V)$. To this aim, similarly to \eqref{eq:norm_1}, we introduce the following norm for $h \in \cac_3(V)$:
\begin{equation}\label{eq:norm_2}
{\|h\|}_{\mu} = \sup_{(s,u,t) \in \mathcal{S}_3([0,T])} \dfrac{|h_{sut}|}{{|t-s|}^{\mu}}.
\end{equation}
Then the $\mu$-H\"older continuous increments in $\cac_3(V)$ are defined as:
\begin{equation*}
\cac_3^{\mu}(V) := \lbrace h \in \cac_{3}(V); {\|h\|}_{\mu} < \infty \rbrace .
\end{equation*}
{Notice that the ratio in \eqref{eq:norm_2} could have been written as $\frac{|h_{sut}|}{|t-u|^{\mu_{1}} |u-s|^{\mu_{2}}}$ with $\mu_{1}+\mu_{2}=\mu$, in order to stress the dependence on $u$ of our increment $h$. However, expression \eqref{eq:norm_2} is simpler and captures the regularities we need, since we are working on the simplex $\cs_{3}$.}

The building block of the {rough path} theory is the so-called sewing map lemma. We recall this fundamental result here.\begin{proposition}\label{prop:La}
	Let $h \in \cac_3^{\mu}(V)$ for $\mu > 1$ be such that $\der h = 0$. Then there exists a unique $g = \Lambda(h) \in \cac_2^{\mu}(V)$ such that $\der g = h$. Furthermore for such an $h$, the following relations hold true:
	\begin{equation*}
	\der \Lambda(h) = h ,
	\quad\text{  and  }\quad
{\|\Lambda h\|}_{\mu} \leq \dfrac{1}{2^{\mu} - 2} {\|h\|}_{\mu} .
	\end{equation*}
\end{proposition}

%\subsection{Elementary computations in {$\mathcal{C}_{2}$ and $\mathcal{C}_{3}$}}
For our considerations below, we shall need some basic rules about products of increments. Towards this aim, let us specialize our setting to the state space $V=\R$. We will also write $\mathcal{C}_{k}^{\mu}$ for $\mathcal{C}_{k}^{\mu}(\R)$. Then $\left(\mathcal{C}_{\ast}, \der\right)$ can be endowed with the following product: for $g \in \mathcal{C}_n$ and $h \in \mathcal{C}_m$ we let $gh$ be the element of $\mathcal{C}_{m+n-1}$ defined by
\begin{equation*}
(gh)_{t_1, \ldots, t_{m+n-1}} = g_{t_1, \cdots, t_n}h_{t_n,\cdots t_{m+n-1}},
\quad\text{  for  }\quad
(t_1, \ldots, t_{m+n-1}) \in \mathcal{S}_{m+n-1}([0,T]).
\end{equation*}
We now label a rule for discrete differentiation of products, which will be used throughout the article. Its proof is an elementary application of the definition \eqref{eq:def-delta}, and is omitted for sake of conciseness.
\begin{proposition}\label{prop:der_rules}
	The following rule holds true:
	Let $g \in \cac_1$ and $h \in \cac_2$. Then $gh \in \cac_2$ and
	\begin{equation*}
	\der(gh) = -\der g\, h + g\, \der h .
	\end{equation*}
\end{proposition}

\noindent
The iterated integrals of smooth functions on $[0,T]$ are particular cases of elements of $\cac_2$, which will be of interest. Specifically, for smooth real-valued functions $f$ and $g$, let us denote $\int f dg$ by $\mathcal{I}(f dg)$ and consider it as an element of $\cac_2$: for $(s,t) \in \mathcal{S}_{2}\left([0,T]\right)$ we set
\begin{equation*}
\mathcal{I}_{st} (f dg) = \left(\int f dg\right)_{st} = \int_{s}^{t} f_u dg_u .
\end{equation*}

\subsection{Weakly controlled processes}\label{sec:weak-contr}

From now on we will focus on a space $V$ of the form $V = \R^d$. The main assumption on the $\R^d$-valued noise $\bzeta$ of equation~\eqref{d1} or~\eqref{d3} is that it gives rise to a geometric rough path. This assumption can be summarized as follows.
\begin{hypothesis}\label{hyp:zeta}
	The path $\zeta:[0,T] \to \R^d$ belongs to the H\"older space ${\cac}^{\al}([0,T];\R^d)$ with $\al > \frac{1}{3}$ and $\zeta_0 = 0$. In addition $\zeta$ admits a L\'evy area above itself, that is, there exists a two index map $\mathbf{\zeta}^2 : {\mathcal{S}_{2}\left([0,T]\right)} \to \R^{d,d}$ which belongs to $\cac_{2}^{2\al}(\R^{d,d})$ and such that 
	\beq\label{eq:zeta^2}
	\der {\mathbf{\zeta}}_{sut}^{2;ij} = {\der \zeta}_{su}^{i} \otimes {\der \zeta}_{ut}^{j},
	\quad\text{  and  }\quad
	{\mathbf{\zeta}_{st}^{2;ij} + \mathbf{\zeta}_{st}^{2;ji}} =  \der \zeta_{st}^{i} \otimes \der \zeta_{st}^{j} .
	\eeq
	The $\al$-H\"older norm of $\zeta$ is denoted by:
	\beq\label{eq:zeta-norm}
	\|\mathbf{\zeta}\|_{\al} = \cn(\zeta;\cac_1^{\al}([0,T],\R^d))
	+\cn(\mathbf{\zeta}^2;\cac_2^{2\al}([0,T],\R^{d,d})).
	\eeq
\end{hypothesis}

\begin{remark}
In this article we restrict our analysis to $\al$-H\"older continuous signals with $\al>1/3$. This assumption is imposed in order to limit Taylor type expansions to reasonable lengths. However, generalizations to lower regularities can be achieved thanks to lengthy additional computations.
\end{remark}

\begin{remark}\label{rmk:fbm}
As mentioned in the introduction, a typical example of path $\zeta$ satisfying Hypothesis~\ref{hyp:zeta} are generic sample trajectories of a fractional Brownian motion $B=(B^{1,H},\ldots,$ $B^{d,H})$ defined on a complete probability space $(\Omega,\cf,\bp)$. In order to ensure that the signal is $\al$-H\"older continuous with $\al>1/3$, we need to restrict the Hurst parameter to be such that $H>1/3$ (see e.g \cite[Chapter 15]{friz-victoir}). All the results stated below apply to this context.
\end{remark}

\noindent
We now define the notion of a weakly controlled process.
\begin{definition}\label{def:weakly-ctrld}
	Let $z$ be a process in $\cac_{1}^{\ka}(\R^n)$ with $\ka \leq \al$ and $2\ka +\ga > 1$. We say that $z$ is weakly controlled by $\zeta$ if $\der z \in \cac_2^{\ka}(\R^n)$ can be decomposed into
	\begin{equation}\label{eq:weakly-controlled}
	\der z^{i} = z^{\zeta;i i_1} \der \zeta^{i_1} + \rho^{i}, 
	\quad\text{   i.e.  }\quad
{\der z}_{st}^{i} = z_{s}^{\zeta;i i_1}{\der \zeta}_{st}^{i_1} + \rho_{st}^{i} ,
	\end{equation}
	for all $(s,t) \in \mathcal{S}_2 \left([0,T]\right)$ and $i=1, \ldots,n$. In the previous formula we assume $z^{\zeta} \in \cac_{1}^{\ka}(\R^{n,d})$ and $\rho$ is a more regular remainder such that $\rho \in \cac_{2}^{2\ka}(\R^{n})$. The space of weakly controlled paths on $[a,b]$ with H\"older continuity $\ka$ and driven by $\zeta$ will be denoted by $\mathcal{Q}_{\zeta}^{\ka}(\R^n)$, and a process $z \in \mathcal{Q}_{\zeta}^{\ka}(\R^n)$ can be considered as a couple $(z, z^\zeta)$. A possible norm on $\mathcal{Q}_{\zeta}^{\ka}(\R^n)$ is given by
	\begin{multline}\label{eq:norm-controlled}
	\cn[z;\mathcal{Q}_{\zeta}^{\ka}(\R^n)] = 
	\cn[z; \cac_1^{\infty}(\R^n)]+\cn[z; \cac_1^{\ka}(\R^n)] + \cn[z^\zeta; \cac_1^{\infty}(\R^{n,d})] \\+ \cn[z^\zeta; \cac_1^{\ka}(\R^{n,d})] + \cn[\rho; \cac_2^{\ka}(\R^n)].
	\end{multline}
\end{definition}

\begin{remark}\label{rem:j-e}
In the sequel we will also need to introduce spaces of the form $\cq_{\zeta}^{\ka}({[0,T]}; V)$, where $V$ is an infinite dimensional functional space. As an example, our notion of rough viscosity solution will require $V = \cac^1(\R^m)$.
\end{remark}

In the following proposition we give an elaboration of the classical composition rule in~\cite{gubinelli}, for controlled processes taking into account time dependent functions. A short proof of this result can be found in~\cite{CHT}.
 
\begin{proposition}\label{prop:smooth_of_weak}
Let $z$ be a controlled process in $\cq_{\zeta}^{\ka}([a,b];\R^n)$, with decomposition \eqref{eq:weakly-controlled} and initial condition $z_0$. Let also $L$ be a function in $\cac_b^{\tau, 2} ([a,b] \times \R^n;\R^{m})$ for $\tau \geq 2\ka$. Set $\hat{z}_t = L(t, z_t)$, with initial condition $\hat{\al} = L(a, z_0)$. Then the increments of $\hat{z}$ can be decomposed into
$$
\der \hat{z}_{st} = \hat{z}_s^{\bzeta} \, \der \zeta_{st} + \hat{\rho}_{st} + (\der L(\cdot, z))_{st},
$$
with the convention that $\nabla_z L(s, z)\in\R^{m,n}$, and $\hat{z}^{\zeta}, \hat{\rho}$ defined by
\beq\label{eq:6-terms}
\hat{z}_s^{\zeta} = \nabla_z L(s, z_s) z_s^{\zeta} \quad \text{and} \quad \hat{\rho}_{st}=\nabla_z L(s,z_s) \rho_{st} +      \lc  {\lp \der L(\cdot ,z) \rp}_{st} -\nabla_z L(s,z_s) \der z_{st} \rc . 
\eeq
Furthermore $\hat{z} \in \cq_{\zeta}^{\ka} ([a,b]; \R^m)$, and
\beq\label{eq:6-a}
\cn \lc \hat{z}; \cq_{\zeta}^{\ka}([a,b];\R^m) \rc \leq c_L \lp 1 + \cn^2\lc z; \cq_{\zeta}^{\ka} ([a,b]; \R^n) \rc \rp.
\eeq
\end{proposition}

One of the main reasons to introduce weakly controlled paths is that they provide a natural set of functions which can be integrated with respect to a rough path. Below we handle mixed integrals with a rough term and a Lebesgue type term, for which we need an additional notation.
%The class of weakly controlled paths provides a natural and basic set of functions which can be integrated with respect to a rough path. %The basic proposition in this direction, whose proof can be found in \cite{Gubinelli}, is summarized below.
\begin{notation}\label{not:leb-int}
Let $\eta$ be a path in $\cac_b^0([a,b];\R)$, where $\cac_b^0([a,b];\R)$ stands for the space of $\R$-valued continuous functions defined on $[a,b]$. For all $a \leq s \leq t \leq b$ we denote by $\ci_{st}(\eta)$ the Lebesgue integral $\int_s^t \eta_u du$.
\end{notation}
We are now ready to state our integration result. Notice again that our formulation of the integral differs from the standard one, due to the fact that we are paying extra attention to the drift term in our control setting. We refer to~\cite{CHT} for a complete proof of this proposition.

\begin{proposition}\label{prop:integral_as_weak}
For a given $\al > \frac{1}{3}$ and $\ka < \al$, let $\zeta$ be a process satisfying Hypothesis~\ref{hyp:zeta}. Furthermore, let $\eta \in \cac_b^0([a,b]; \R)$ as in Notation~\ref{not:leb-int}, and $\mu \in \cq_{\zeta}^{\ka}([a,b]; \R^d)$ as given in Definition~\ref{def:weakly-ctrld} whose increments can be decomposed as
\beq\label{b1}
\der \mu^i = \mu^{\zeta;i i_1} \der \zeta^{i_1} + \rho^{\mu;i}, \quad\text{ where } \mu^{\zeta} \in \cac_1^{\ka}([a,b]; \R^{d,d}), \, \rho^{\mu} \in \cac_2^{2\ka}([a,b]; \R^d).
\eeq
Define $z$ by $z_a = y \in \R$ and 
\beq\label{eq:b}
\der z = \mu^i\der \zeta^i + \mu^{\zeta;i i_1} \mathbf{\zeta}^{2;i_1 i} + \ci(\eta) + \Lambda(\rho^{\mu;i} \der \zeta^i + \der \mu^{\zeta;i i_1} \mathbf{\zeta}^{2; i_1 i}),
\eeq
where the term $\ci(\eta)$ is introduced in Notation~\ref{not:leb-int} and $\laa$ is the sewing map in Proposition~\ref{prop:La}. Then:
\begin{enumerate}[wide, labelwidth=!, labelindent=0pt, label=\textnormal{(\arabic*)}]
\setlength\itemsep{.05in}

\item 
 The path $z$ is well-defined as an element of $\cq_{\zeta}^{\ka}([a,b]; \R)$, and $\der z_{st}$ coincides with the Lebesgue-Stieltjes integral $\int_s^t \mu_u^i d\zeta_u^i + \int_s^t \eta_u du$  whenever $\zeta$ is a differentiable function.

\item The semi-norm of $z$ in $\cq_{\zeta}^{\ka} ([a,b]; \R)$ can be estimated as 
\begin{multline}\label{eq:6-b}
\cn[z; \cq_{\zeta}^{\ka}([a,b]; \R)] \\
\leq c_{\zeta} \lcl 1 + \|\mu_a\| + (b-a)^{\al-\ka} \lp \|\mu_a\| + \cn[\mu; \cq_{\zeta}^{\ka}([a,b]; \R^d)] + \cn[\eta; \cac_b^0([a,b]; \R)] \rp \rcl,
\end{multline}
where the constant $c_{\zeta}$ verifies $c_{\zeta} \leq c(|\zeta|_{\al} + |\mathbf{\zeta}^2|_{2\al})$ for a universal constant $c$.

\item It holds that 
\beq\label{eq:riemann}
\der z_{st} = \lim_{|\Pi_{st}| \to 0} \sum_{q=0}^n \lc \mu_{t_q}^i \der \zeta_{t_q t_{q+1}}^i + \eta_{t_q} (t_{q+1}-t_q) + \mu_{t_q}^{\zeta;i i_1} \mathbf{\zeta}_{t_q, t_{q+1}}^{2; i_1 i} \rc,
\eeq
for any $a \leq s \leq t \leq b$. In \eqref{eq:riemann} the limit is taken over all partitions $\Pi_{st} = \{s=t_0, \ldots, t_n=t\}$ of $[s,t]$, as the mesh of the partition goes to zero. 
\end{enumerate}
\end{proposition}

\subsection{Strongly controlled processes}\label{sec:strongProcess}

In the sequel, some of our computations will require a finer description than~\eqref{eq:weakly-controlled} for controlled processes. In this section we introduce this kind of decomposition and we will also derive an application to rough integration.
Let us start with the definition of strongly controlled process. 
\begin{definition}\label{def:str-contr}
Similarly to Definition~\ref{def:weakly-ctrld}, consider a process $\nu \in \cac_1^{\ka}(\R^n)$ with $\ka \leq \al$ and $2\ka+\al > 1$. Let us suppose that the increment $\der \nu$ lies in $\cac_2^{\ka}$ and can be decomposed as 
\beq\label{eq:str-contr-a}
\der \nu^j = \nu^{\zeta;j i_1} \der \zeta^{i_1} + \nu^{\bzeta^{2}; j i_1 i_2} \bzeta^{2; i_2 i_1} + \rho^{\nu;j},
\eeq
where the regularity for the paths $\nu^{\zeta}$, $\nu^{\bzeta^2}$ and $\rho^{\nu}$ are respectively
\beq\label{eq:str-contr-b}
\nu^{\zeta}, \nu^{\bzeta^2} \in \cac_1^{\ka}([a,b]), \quad\text{and}\quad \rho^{\nu} \in \cac_2^{3\ka}([a,b]).
\eeq
In addition, we assume that $\nu^{\zeta}$ is also a weakly controlled process as in Definition~\ref{def:weakly-ctrld}, with decomposition
\beq\label{eq:str-contr-l}
\der \nu^{\zeta; ji_1} = \nu^{\zeta^{2}; j i_1 i_2} \der \zeta^{i_2} + \rho^{\nu^{\zeta};j},
\eeq
where $\nu^{\zeta^{2}; j i_1 i_2}$ is like in \eqref{eq:str-contr-a} and $\rho^{\nu^{\zeta}; j}$ is a remainder in $\cac_2^{2\ka}([a,b])$. Then we say that $\nu$ is a strongly controlled process in $[a,b]$. We call $\tilde{\cq}_{\zeta}^{\ka}(\R^n)$ this set of paths, equipped with the semi-norm
$$
\cn[\nu; \tilde{\cq}_{\zeta}^{\ka}(\R^n)] = \cn[\nu;\cac_1^{\ka}] + \cn[\nu^{\zeta};\cac_1^{\infty}] + \cn[\nu^{\zeta^{2}}; \cq_{\zeta}^{\ka}] + \cn[\rho;\cac_2^{3\ka}].
$$
\end{definition}

On our way to the definition of viscosity solutions, we will encounter integrals of controlled processes with respect to other controlled processes. We recall how to define this kind of integral in the proposition below, which is a slight elaboration of \cite[Theorem~1]{gubinelli} and requires the introduction of strongly controlled processes. 

\begin{proposition}\label{prop:integral_controlled}
For a given $\al > \frac{1}{3}$ and $\ka < \al$, let $\zeta$ be a process satisfying Hypothesis~\ref{hyp:zeta}. 
Given $m,n \geq 1$, we consider a weakly controlled process $\mu \in \cq_{\bzeta}^{\ka}([a,b]; \R^m)$ as introduced in Definition~\ref{def:weakly-ctrld} and a strongly controlled process $\nu \in \tilde{\cq}_{\zeta}^{\ka}([a,b]; \R^n)$ as given in Definition~\ref{def:str-contr}. Specifically, for every $i=1, \ldots , m$ and $j=1, \ldots, n$ we assume the decomposition
\begin{align}
\der \mu^i &= \mu^{\zeta;i i_1} \der \zeta^{i_1} + \rho^{\mu; i}, \label{eq:int_contr-g} \\ %\text{ where } \mu^{\zeta} \in \cac_1^{\ka}([a,b]; \R^{m,d}), \, \rho^{\mu} \in \cac_2^{2\ka}([a,b]; \R^m),\\
\der \nu^j &= \nu^{\zeta;j i_1} \der \zeta^{i_1} + \nu^{\bzeta^2; j i_1 i_2}\bzeta^{2; i_2 i_1} +  \rho^{\nu;j}, \label{eq:int_contr-h}%\text{ where } \nu^{\zeta} \in \cac_1^{\ka}([a,b]; \R^{n,d}), \, \rho^{\nu} \in \cac_2^{2\ka}([a,b]; \R^n).
\end{align}
where the regularities for the paths $\mu^{\bzeta}, \mu^{\bzeta^2}, \rho^{\mu}, \nu^{\bzeta}, \nu^{\bzeta^2}, \rho^{\nu}$ are respectively
\begin{equation*}
\mu^{\bzeta; i i_1}, \nu^{\bzeta; j i_1}, \nu^{\bzeta^2; j i_1 i_2} \in \cac_1^{\ka}([a,b]), \qquad
\rho^{\mu; i} \in \cac_2^{2 \ka} ([a,b]), \qquad
\rho^{\nu; j} \in \cac_2^{3 \ka} ([a,b]).
\end{equation*}
Furthermore, let $\eta \in \cac_1^0([a,b]; \R^{m, n})$. For $1 \leq i \leq m$ and $1\leq j \leq n$ we define a controlled process $z^{ij}$ by setting $z_a^{ij} = y^{ij} \in \R$ and a decomposition
\beq\label{eq:int_contr-e}
\der z_{st}^{ij} = \cg_{st}^{ij} + \ci_{st}(\eta^{ij}) + \laa_{st}\lp \cj^{ij} \rp,
\eeq
where the term $\ci(\eta^{ij})$ is introduced in Notation~\ref{not:leb-int}, where the increment $\cg^{ij}$ is defined by 
\beq\label{eq:int_contr-i}
\cg_{st}^{ij} = \mu_s^i \nu_s^{\bzeta;j i_1} \der \bzeta_{st}^{i_1} + \lp \mu_s^i \nu_s^{\bzeta^2; j i_1 i_2} + \mu_s^{\bzeta; i i_2} \nu_s^{\bzeta; j i_1} \rp \bzeta_{st}^{2; i_2 i_1},
\eeq
and where $\cj^{i j}$ is an increment in $\cac_3^{3\ka}$ satisfying
\beq\label{eq:int_contr-f}
\cj_{sut}^{i j} = - \der \cg_{sut}^{ij}, 
\quad \text{and} \quad
\cn[\cj^{ij}; \cac_3^{3\ka}([a,b])] \leq c_{\bzeta} \cn[\mu^i; \cq_{\zeta}^{\ka}([a,b])] \cn[\nu^j; \tilde{\cq}_{\zeta}^{\ka}([a,b])].
\eeq
%where we recall the notation $\der$ in \eqref{eq:der}.
Then similarly to Proposition~\ref{prop:integral_as_weak}, the following assertions hold true:
\begin{enumerate}[wide, labelwidth=!, labelindent=0pt, label=\textnormal{(\arabic*)}]
\setlength\itemsep{.02in}

\item 
The path $z$ is well-defined as an element of $\cq_{\zeta}^{\ka}([a,b]; \R^{m,n})$, and $\der z_{st}$ coincides with the Lebesgue-Stieltjes integral $\int_s^t \mu_u \otimes d\nu_u + \int_s^t \eta_u du$  whenever $\zeta$ is a differentiable function.

%\item The semi-norm of $z$ in $\cq_{\zeta}^{\ka} ([a,b]; \R)$ can be estimated as 
%\begin{multline}\label{eq:6-b}
%\cn[z; \cq_{\zeta}^{\ka}([a,b]; \R)] \\
%\leq c_{\zeta} \lcl 1 + \|\mu_a\| + (b-a)^{\al-\ka} \lp \|\mu_a\| + \cn[\mu; \cq_{\zeta}^{\ka}([a,b]; \R^d)] + \cn[\eta; \cac_1^0([a,b]; \R)] \rp \rcl,
%\end{multline}
%where the constant $c_{\zeta}$ verifies $c_{\zeta} \leq c(|\zeta|_{\al} + |\mathbf{\zeta}^2|_{2\al})$ for a universal constant $c$.

\item It holds that 
\begin{equation}\label{eq:riemann-2}
\der z_{st}^{ij} = \lim_{|\Pi_{st}| \to 0} \sum_{q=0}^n \lc \mu_{t_q}^i  \nu_{t_q}^{\bzeta;j i_1} \der \bzeta_{t_q t_{q+1}}^{i_1} 
+ \lp \mu_{t_q}^i \nu_{t_q}^{\bzeta^2;j i_1 i_2} 
+ \mu_{t_q}^{\bzeta; i i_1} \nu_{t_q}^{\bzeta; j i_2} \rp \bzeta_{t_q t_{q+1}}^{2; i_2 i_1} +\eta_{t_q}^{ij} (t_{q+1} - t_q) \rc \, ,
\end{equation}
for any $a \leq s \leq t \leq b$. In \eqref{eq:riemann-2} the limit is taken over all partitions $\Pi_{st} = \{s=t_0, \ldots, t_n=t\}$ of $[s,t]$, as the mesh of the partition goes to zero. 
\end{enumerate}
\end{proposition}

\begin{remark} \label{rmk: time-decomp}
    Consider a path $z$ with decomposition (\ref{eq:int_contr-e})-(\ref{eq:int_contr-i}). The decomposition of its time increments can be written as
    \begin{equation} \label{eq: time-decomp}
        \delta z_{st}=z_{s}^{\zeta}\delta \zeta_{st}+z_{s}^{\zeta^{2}}\zeta^{2}_{st}+z_{s}^{\tau}\delta \tau _{st}+R_{st},
    \end{equation}
    where we have set $\delta \tau _{st}=t-s$, with
    \[z_{s}^{\zeta}=\mu_{s}\otimes \nu_{s}^{\zeta}, 
    \qquad z_{s}^{\zeta^{2}}=\mu_{s}\otimes\nu_{s}^{\zeta^{2}}+\mu_{s}^{\zeta}\otimes \nu_{s}^{\zeta},
    \qquad z^{\tau}=\eta,\]
    and where the remainder $R$ sits in $C^{3\kappa}(\R^{m})$. This type of decomposition is even more detailed than (\ref{eq:str-contr-a}). It will be used in later computations, when a proper drift identification is necessary.
\end{remark}
\begin{definition}\label{def:time_deriv}
    In the context of Remark \ref{rmk: time-decomp}, whenever a path $z$ admits the decomposition~(\ref{eq: time-decomp}), we will write
    \begin{equation}
      \label{eq:time-not}  z^{\tau}\equiv \frac{dz}{d\tau}=\dot z.
    \end{equation}
    Otherwise stated, for a controlled process $z^{\tau}\in \mathcal{Q}^{\ka}(\R^{m})$, one can define a rough time derivative $z^{\tau}=\dot{z}$ if there exists two paths $z^{\zeta}$, $z^{\zeta^{2}}$ such that the increment $\delta z_{st}-z_{s}^{\zeta}\delta \zeta_{st}-z_{s}^{\zeta^{2}}\zeta^{2}_{st}-z_{s}^{\tau} \delta \tau_{st}$ belongs to $\mathcal{C}^{3\kappa}(\R^{m})$.
\end{definition}
We close this section by spelling out a technical result which will be used for our HJB equations. We shall thus consider controlled processes which are indexed by an additional spatial parameter $\theta$. We give a composition rule for such processes. 

\begin{proposition}\label{prop:str_comp}
For a given $\al > \frac{1}{3}$ and $\ka < \al$, let $\bzeta$ be a process satisfying Hypothesis~\ref{hyp:zeta}. Consider two $\cac^3(\R^m;\R^m)-$valued strongly controlled processes $\mu$ and $\nu$ admitting a decomposition of type~\eqref{eq:str-contr-a}:
\begin{align}\label{eq:mu_nu_str}
\der \mu^j(\theta) &= \mu^{\bzeta; j i_1}(\theta) \der \bzeta^{i_1} + \mu^{\bzeta^2; ji_1 i_2}(\theta) \bzeta^{2;i_2 i_1} + \rho^{\mu;j}(\theta) \nonumber\\
\der \nu^j(\theta) &= \nu^{\bzeta;ji_1}(\theta) \der \bzeta^{i_1} + \nu^{\bzeta^2; ji_1 i_2}(\theta) \bzeta^{2; i_2 i_1} + \rho^{\nu; j}(\theta),
\end{align}
for $\theta \in \R^m$. Then $\mu \circ \nu$ is another strongly controlled process with decomposition 
\begin{align}
{\lc \mu \circ \nu \rc}^{\bzeta; j i_1}(\theta) &= \mu^{\bzeta; j i_1}(\nu(\theta)) + \pt_{\eta^k} \mu^j(\nu(\theta)) \nu^{\bzeta; ki_1}(\theta) \label{eq:mu_nu_str_decomp-i}\\
{\lc \mu \circ \nu \rc}^{\bzeta^2; j i_1 i_2} (\theta) &= \mu^{\bzeta^2; j i_1 i_2}(\nu(\theta)) + \pt_{\eta^k} \mu^j(\nu(\theta)) \nu^{\bzeta^2; k i_1 i_2}(\theta) + \pt_{\eta^k} \mu^{\bzeta; j i_2}(\nu(\theta)) \nu^{\bzeta; k i_1}(\theta)\nonumber \\ 
&+ \pt_{\eta^k} \mu^{\bzeta; j i_1}(\nu(\theta)) \nu^{\bzeta; k i_2}(\theta) 
+\pt_{\eta^{k_1} \eta^{k_2}}^2 \mu^j (\nu(\theta)) \nu^{\bzeta; k_1 i_1}(\theta) \nu^{\bzeta; k_2 i_2}(\theta) \label{eq:mu_nu_str_decomp-ii}.
\end{align}
\end{proposition}

\subsection{Rough differential equations with relaxed controls}
This section is devoted to recall some of the results concerning controlled equations derived in \cite{CHT}. This will lay the ground for our developments throughout the article. Let us then recall that one starts from a rough path $\bzeta$ as in Hypothesis~\ref{hyp:zeta}. 
%One also considers a compact subset $\cv \subset\subset \R^d$ and a measure-valued (relaxed) control $\ga$ such that
As far as the relaxed control is concerned, our state space will be the following:
\begin{definition}\label{def:rel-control}
Consider a domain $U \subset \R^d$. Then the state space for all the relaxed controls below will be a compact subset $\ck$ of $\cp(U)$, where $\cp(U)$ denotes the set of probability measures on $U$. Notice that one can take $\ck = \cp(U)$ whenever $U$ is a compact domain of $\R^d$.
\end{definition}
We now state our basic assumptions on relaxed controls, which will prevail throughout the article.
\begin{hypothesis}\label{hyp:rel-control-1}
In the sequel we consider some measure-valued paths $\{ \ga_s; 0 \leq s \leq T \}$ such that
\begin{enumerate}[wide, labelwidth=!, labelindent=0pt, label=\textnormal{(\arabic*)}]
\setlength\itemsep{.05in}
\item For every $s$, we have $\gamma_{s}\in \mathcal{K}$, where $\mathcal{K}$ is introduced in Definition \ref{def:rel-control}.
\item The application $s\mapsto \gamma_{s}$ is measurable from $[0,T]$ to $\mathcal{K}$.
\end{enumerate}
\end{hypothesis}

\noindent
Next for an initial condition $y \in \R^m$, an initial time $s \in [0,T]$, a relaxed control $\ga$ and our rough path $\bzeta$, let $x^{s,y,\ga}$ be the rough differential equation defined by
\beq\label{eq:relaxed-rde}
x_t^{s,y,\ga} = y+\int_s^t \int_U b(r, x_r^{s,y,\ga},a)\ga_r(da)dr + \int_s^t \si(r, x_r^{s,y,\ga}) \, d \bzeta_r .
\eeq
In \eqref{eq:relaxed-rde}, the assumptions on the coefficients $b$ and $\si$ are rather standard within the rough path theory. They are summarized as follows:
\begin{hypothesis}\label{hyp:coeff}
	We assume that the coefficients $b$ and $\sigma$ in \eqref{eq:relaxed-rde} satisfy $b \in \mathcal C_b^{0,1,1}([0,T] \times \R^m \times U; \R^m)$ and $\sigma \in \mathcal C_b^{1,3}([0,T]\times \R^m; \R^{m,d})$. \emph{Furthermore $b(u, x, a)$ is uniformly continuous in $u$, and it is also assumed to be Lipschitz in space uniformly in time and control.}\end{hypothesis}
We now turn to a slight elaboration of Proposition 3.3 from \cite{CHT}, giving existence and uniqueness for RDEs with relaxed controls.
\begin{proposition}\label{ExistUniq2}
Consider $\al > \frac{1}{3}$ and a rough path $\bzeta$ fulfilling Hypothesis~\ref{hyp:zeta}. The coefficients $b$ and $\si$ are such that Hypothesis~\ref{hyp:coeff} is satisfied. In addition, $\ga \colon [0,T]\to \mathcal{P}(U)$ is a measure-valued process fulfilling Hypothesis \ref{hyp:rel-control-1}. Then equation~\eqref{eq:relaxed-rde} admits a unique solution in the space $\cq_{\bzeta}^{\ka}$ introduced in Definition~\ref{def:str-contr}, for all $\frac{1}{3} < \ka < \ga$.
\end{proposition}
\begin{proof}
    Our proposition establishes existence and uniqueness under weaker assumptions on the control $\gamma$ and the drift term $b$ (compared to Proposition 3.3 and Corollary 3.11 from \cite{CHT}). More specifically, the H\"older assumptions on $s\mapsto \gamma_{s}$ assumed in \cite{CHT} is avoided here. We will need this variant for the perturbative approach developed in the upcoming section. The interested reader can check that the proof of Proposition 3.3 and Corollary 3.11 in \cite{CHT} are unchanged under those weaker hypotheses. We skip the details for the sake of brevity.\end{proof}

\subsection{Value function and rough HJB equations}

A significant part of reinforcement learning in continuous time can be reduced to an optimization problem over strategies. In the context of a rough dynamics like~\eqref{eq:relaxed-rde}, the goal is then to optimize over $\ga$ a reward function given by
\beq\label{eq:J}
J_{s,T} (\ga, y) = \int_s^T F(r, x_r^{\ga}, \ga_r) dr + G(x_T^{\ga}),
\eeq
where $F$ and $G$ are functions whose regularity can be specified as follows:
\begin{hypothesis}\label{hyp:reward}
	$(t,x,\gamma) \mapsto F(t,x,\gamma)$ is bounded, Lipschitz continuous in $x$, uniformly over $\gamma \in \mathcal K$ and Lipschitz in $\gamma \in \mathcal K$. The function $G:\mathbb R^m \to \mathbb R$ is bounded and continuous.
\end{hypothesis}

\begin{remark}\label{rmk: F}
As mentioned in \cite[Remark~4.3]{CHT}, under certain non-degeneracy conditions on $\ga$ a typical $F$ verifying Hypothesis~\ref{hyp:reward} is given by
\beq\label{eq:F-ex}
F(s,x,\ga) = \int_U R(s, x, u) \dot{\ga}_s(u) du - \la \int_U \dot{\ga}_s(u) \log \dot{\ga}_s(u) du.
\eeq
In \eqref{eq:F-ex} notice that $\dot{\ga}_s$ stands for the density of $\ga_s$, and the term $\int \dot{\ga} \log \dot{\ga}$ is an entropy term favoring the exploration of $U$. The quantity $\la > 0$ has to be understood as a regularization parameter. See more about this type of coefficient in the forthcoming Remark \ref{rmk: K}. 
\end{remark}

Before describing our notion of value function, let us label an assumption, borrowed from~\cite{CHT}, on the relaxed control $\ga$. This hypothesis is more restrictive than Hypothesis~\ref{hyp:rel-control-1}. It implies that $\ga$ lies in a compact domain, in order to make sure that the supremum in the definition of the value function is attained. We recall those assumptions here:
\begin{hypothesis}\label{hyp:rel-control}
In the sequel we consider measure-valued paths $\{ \ga_s; 0 \leq s \leq T \}$ such that
\begin{enumerate}[wide, labelwidth=!, labelindent=0pt, label=\textnormal{(\arabic*)}]
\setlength\itemsep{.05in}
\item For every $s$, we have $\ga_s \in \mathcal{K}$, where $\mathcal{K}$ is the set introduced in Definition \ref{def:rel-control}.
\item We have $\ga \in \cv^{\ep, L}$ for two fixed constants $\ep, L > 0$, where the set $\cv^{\ep, L}$ is defined by
\beq\label{eq:cv^ep,L}
\cv^{\ep, L} = \lcl \ga: [0,T] \mapsto \ck ; \,  W_2(\ga_s, \ga_t) \leq L|t-s|^{\ep} \rcl,
\eeq
and $W_2$ stands for the 2-Wasserstein distance for probability distributions in $\cp(U)$. Also recall that the $W_{2}$-distance between two measures $m_{1}, m_{2}$ is defined by
$$
W_2(m_1, m_2) = \inf \lcl {\lp\int \int d^2(x,y) \, dm(x,y) \rp}^{\frac{1}{2}}; m \in \gga(m_1, m_2) \rcl,
$$
with $\gga(m_1, m_2)$ denoting the set of couplings with marginals $m_1$, $m_2$.
\end{enumerate}
\end{hypothesis}
Towards the optimization of the action $\ga$ seen as a relaxed control, we set
\beq\label{eq:V_sy}
V(s,y) := \sup \lcl J_{s,T}(\ga, y): \ga \in \cv^{\ep, L}([s,T]; \ck) \rcl,
\eeq
where we recall that $\cv^{\ep, L}$ is the set introduced in \eqref{eq:cv^ep,L}. Our main result in \cite{CHT} was the derivation of a rough HJ equation for $V$:
\begin{equation}\label{d3}
\partial_t v(t,y) + \sup_{\gamma \in \mathcal P(U)} \left\{ \nabla v(t,y) \cdot \int b(t,y,a)\gamma(da) + F(t,y,\gamma) \right\} + \nabla v(t,y) \cdot \sigma(t,y)  \dot{\zeta}_t = 0 \, ,
\end{equation}
for $(t,y)\in[0,T]\times \R^n$, with terminal condition $v(T,y)=G(y)$.
We briefly recall our findings in this section.

%The dynamic programming principle implies that the optimal value function,
%\begin{align}\label{eq:value}
%	V(s,y) := \sup \left\{ J_{sT}(\gamma,y) : \gamma \in \mathcal K\right\},~\quad J_{sT}(\gamma,y) = \int_s^T F(r,x_r^\gamma, \gamma_r) dr + G(x_T^\gamma),
%\end{align}
%satisfies 
%\begin{align}
%	V(s_1,y) = \sup \left\{ \int_{s_1}^{s_2} F(r,x_r^{s_1,y,\gamma},\gamma_r) dr + V(s_2,x_{s_2}^{s_1,y,\gamma}) : \gamma \in \mathcal K \right\}
%\end{align}
%where
%\begin{align}\label{eq:rde}
%	x_t^{s,y,\gamma} = y + \int_s^t \int_U b(r,x_r^{s,y,\gamma},a)\gamma_r(da) dr + \int_s^t \sigma(r,x_r^{s,y,\gamma} d\zeta_r.
%\end{align}
%
%We assume that the following hypotheses are enforced.

First, consider the following definition of a rough test function. This definition parallels the definition of 
$\mathcal C^1$ test functions considered in the deterministic theory, albeit perturbed by a noisy term involving $\zeta$.

\begin{definition}\label{def:test}
	Let $\sigma$ be a coefficient such that Hypothesis~\ref{hyp:coeff} is fulfilled. We consider a path $\zeta \in \mathcal C^\alpha$ with $\alpha > 1/3$ verifying Hypothesis~\ref{hyp:zeta}. We say that $\psi \in \mathcal C^{\alpha,2}$ is an element of the set of test functions $\mathcal T_\sigma$ if there exists a drift term $\psi^t \in \mathcal C_b^{0,1}$ such that $\psi$ satisfies the following linear rough PDE:
	\begin{align}
		\delta\psi_{s_1,s_2}(y) = \int_{s_1}^{s_2} \psi_r^t(y) - \int_{s_1}^{s_2} \nabla \psi_r(y) \cdot \sigma^k(r,y) d\zeta_r^k,
	\end{align}
	for all $y\in \mathbb R^m$ and $0 \leq s_1\leq s_2\leq T$.
\end{definition}

The notion of a rough viscosity solution advocated in \cite{CHT} can now be spelled out as follows.
\begin{definition}\label{def:rough-visc}
	Let $\zeta \in \mathcal C^\alpha$ fulfilling Hypothesis~\ref{hyp:zeta}. The coefficients $b$ and $\sigma$ are assumed to satisfy Hypothesis~\ref{hyp:coeff} as before. Consider a path $v \in \mathcal Q_\zeta^\kappa$. Recall that the set $\mathcal T_\sigma$ is introduced in Definition~\ref{def:test}. We say that $v$ is a rough viscosity supersolution (resp. subsolution) of equation~\eqref{d3} if
	\[
		v_T(y) \geq (\text{resp.} \leq) G(y),
	\]
	and for every element $\psi \in \mathcal T_\sigma$ such that $v-\psi$ admits a local minimum (resp. maximum) at $(s,y)$, the drift $\psi^t$ satisfies
	\[
		\psi^t_s(y) \leq (\text{resp.} \geq) - \sup_{\gamma \in \mathcal K} H(s,y,\gamma,\nabla\psi_s(y)),
	\]
	where we recall that the Hamiltonian $H$ is 
	\begin{equation} \label{eq:Hamil}
		H(t,y,\gamma,p) = p \cdot \int_U b(t,y,a) \gamma(da) + F(t,y,\gamma).
	\end{equation}
The path	$v$ is a rough viscosity solution if it is both a rough viscosity supersolution and subsolution.
\end{definition}
Theorem~4.22 in~\cite{CHT} shows that the value function is the unique {\it rough} viscosity solution of the HJ equation. 
\begin{theorem}\label{thm:HJBequation}
	Consider a coefficient $\sigma \in \mathcal C_b^{1,3}([0,T]\times \mathbb R^m;\mathbb R^{m,d})$ as well as a reward function $F$ satisfying Hypothesis~\ref{hyp:reward}. The path $\zeta$ is assumed to satisfy Hypothesis~\ref{hyp:zeta}. Then, the HJB equation~\eqref{d3} admits a unique rough viscosity solution in the sense of Definition~\ref{def:rough-visc}. This solution is given by the value function $V$ introduced in \eqref{eq:V_sy}.
\end{theorem}

In \cite[Proposition 4.4]{CHT} it was shown that an optimal pair $(\overline{x},\overline{\gamma})$, corresponding to the value $V$ in Theorem~\ref{thm:HJBequation}, exists. We explicitly lay it out here for future use.

\begin{theorem}\label{thm:Optimal Pair} Let $F,G$ be functions satisfying Hypothesis \ref{hyp:reward}. Assume that Hypothesis~\ref{hyp:coeff} is fulfilled, so that Proposition \ref{ExistUniq2} holds. The space $\mathcal{K}$ is the subspace of $\mathcal{P}(U)$ mentioned in Definition \ref{def:rel-control}, and $\mathcal{V}^{\varepsilon,L}$ is the set of H\"older paths given in \eqref{eq:cv^ep,L}.  Then there exists a (non necessarily unique) $\overline{\gamma}\in \mathcal{V}^{\varepsilon,L}$ achieving the maximum $V(0,y)$ defined in (\ref{eq:V_sy}). The pair $(\overline{x},\overline{\gamma})$ is called an optimal pair, where $\overline{x}$ is the unique solution of (\ref{d1}) with the relaxed control taken as $\overline{\gamma}$. 
    
\end{theorem}

\section{Pontryagin Maximum Principle: Relaxed Control Setting}\label{sec: Pontry}

The Pontryagin maximum principle in the classical setting is a necessary condition for the optimality of a dynamics-control pair. Furthermore, it is seen as an equation describing the evolution of Lagrange multipliers related to the optimization problem (\ref{eq:V_sy}). In this section we derive this principle in our rough context. It will also be a way to introduce techniques which will be relevant for upcoming computations. 

\subsection{Statement of Theorem and Strategy}
We directly start by stating the main theorem of this section. As mentioned above, it gives a variational characterization of an optimal dynamics-control pair.
\begin{theorem}\label{Pontryagin}
			Let the assumptions of Theorem \ref{thm:HJBequation} and Theorem \ref{thm:Optimal Pair} prevail.  In addition, $F$ and $G$ are modified to be in $C_{b}^{0,1,0}([0,T]\times \R^m \times \mathcal{P}(U);\R)$ and $C_{b}^{1}(\R^m;\R)$ respectively. Suppose that $(\overline{x},\overline{\gamma})$ is an optimal pair (which exists thanks to Theorem \ref{thm:Optimal Pair}). Also, let $p\colon [0,T]\to \R^m$ be the unique solution of the following backward RDE:
	\begin{equation}\label{RDEp}
				\begin{cases}
					-dp_{t}=\left(\int_{U}Db(t,\overline{x}_{t},a)\overline{\gamma}_{t}(da)\right)^{\text{T}}p_{t}\,dt+D\sigma(t,\overline{x}_{t})p_{t}\,d\zeta_{t} +DF(t,\overline{x}_{t},\overline{\gamma}_{t})\,dt\\
					\hspace{0.43cm}p_{T}=DG(x_{T})
				\end{cases}
			\end{equation}
			Then, one has
			\begin{multline}\label{PMP}
				\left<\int_{U}b(t,\overline{x}_{t},a)\overline{\gamma}_{t}(da),p_{t} \right>+F(t,\overline{x}_{t},\overline{\gamma}_{t})\\
				=
				\sup_{\gamma \in \mathcal{P}(U)}\left\{ \left<\int_{U}b(t,\overline{x}_{t},a)\gamma(da),p_{t} \right> + F(t,\overline{x}_{t},\gamma)  \right\}.
			\end{multline}
		\end{theorem}
        \begin{remark}
            The assumptions on $F$ and $G$ can be relaxed to Assumptions $(D1)$-$(D2)$-$(D3)$ from \cite[p.~102]{yong-zhou}. We have refrained to do so for sake of simplicity.
        \end{remark}
        \begin{remark} Let us comment on the dimensions and definitions of the coefficients in equation~(\ref{RDEp}). As mentioned in Hypothesis \ref{hyp:coeff}, the function $\sigma$ is defined as 
        \[\sigma \colon [0,T]\times\R^{m}\longrightarrow \R^{m}\otimes\R^{d},\]
        where we write the set of matrices $\R^{m,d}$ as a tensor $\R^{m}\otimes \R^{d}$. With this notation in hand one can define $D\sigma$ as a map
        \[D\sigma\colon [0,T]\times \R^{m}\longrightarrow\R^m \otimes \R^{d}\otimes\R^m,
        \quad\text{with}\quad
        [D\sigma (t,x)]^{ijk}=\partial_{x_{k}}\sigma^{ij}(t,x).
        \]
Products are then understood in the usual way. For instance, for $p\in \R^{m}$ we have
        \begin{equation}
            D\sigma (t,x)p\in \R^{m}\otimes \R^{d},
        \quad\text{with}\quad
[D\sigma(t,x)p]^{ij}=\sum_{k=1}^{m}[D\sigma(t,x)]^{ijk}p^{k} .
        \end{equation}
In the same way, for all $(t,x,a)\in [0,t]\times \R^{m}\times U$, we have $Db(t,x,a)\in \R^{m}\otimes \R^{m}$.
            \end{remark}
        \begin{remark}
            In order to make Pontryagin's principle of practical interest, one should plug equation (\ref{RDEp}) into a dynamics for Lagrange type multipliers. We have not pursued this goal here, since our practical considerations will rely on other methods.
        \end{remark}
        As mentioned in the introduction, our method of proof for Theorem \ref{Pontryagin} is based on differentiation procedures along perturbations of the optimal couple $(\overline{x},\overline{\gamma})$. Those perturbations, called spike variations in the literature, are defined as follows:
        \begin{definition} \label{def:spike variation}
          Suppose that $(\overline{x},\overline{\gamma})$ is an optimal pair. Consider a time variable $t_{0}\in[0,T)$. For an additional variable $\beta>0$, define the interval $I=I_{\beta}=[t_{0},t_{0}+\beta]$ and assume $I\subset [0,T]$. Then we define the spike variation $\gamma^{\beta}$ of the optimal relaxed control as
          \begin{equation}\label{spike}\gamma^{\beta}_{t}=\begin{cases} \mu \hspace{0.2cm} \text{ if } \hspace{0.2cm} t\in I\\
			    \overline{\gamma}_{t} \hspace{0.2cm} \text{ if } \hspace{0.2cm} t\notin I
			\end{cases}\end{equation}
            where $\mu$ is a constant probability measure sitting in $\ck$ (recall that $\ck$ is the compact set of $\mathcal{P}(U)$ introduced in Definition~\ref{def:rel-control}).
           \end{definition}
           \begin{remark}
               Defining the perturbations like in ($\ref{spike}$) is one way to circumvent the fact that $\mathcal{P}(U)$ is not a vector space (which prevents perturbations of the form $\overline{\gamma}_{t}+\beta \mu$). Another advantage of spike variations, with respect to other variational forms like flat derivatives (see e.g~\cite[Appendix]{LMS} or~\cite[Chapter 5]{CDL} for an overview of differentiation notions for measures), is that spike variations entail a local information in time. This is precious for optimization purposes.
               \end{remark}
               \begin{remark}
                   Theorem \ref{Pontryagin} is an analog of \cite[Theorem~7]{diehl} to the relaxed case. The proof of the latter theorem hinged on using rough flow transformations to reduce the noisy problem to a deterministic problem with random coefficients. Although we could have invoked the same strategy here, we have resorted to a different approach. Specifically, we use the differentiability of the solutions of the RDEs with respect to the perturbation, and make use of known estimates of the Jacobian of the flow map. This strategy will also feature in our section on the $Q$-function. 
               \end{remark}

\subsection{Differentiating along spike variations}\label{sec: Prelim}

In this section we establish preliminary results concerning differentiation along paths driven by the spike variation \eqref{spike}.
We shall state and prove three Lemmas related to the behaviors of $x^{\beta}$, $\overline{x}$, and $Y^{\beta}$. Before doing so, we will introduce a slight variation of some well known terminology and results found in \cite[Chapter 11 and Chapter 13]{friz-victoir}.

\begin{proposition}\label{prop: Flow and Jacobian}
    Let us suppose that the coefficients $b$ and $\sigma$ satisfy Hypothesis \ref{hyp:coeff}. We also assume that $\zeta$ is a rough path fulfilling Hypothesis \ref{hyp:zeta}. The solution of the RDE (\ref{eq:relaxed-rde}) with initial condition $y$ is given by a flow map $x_{t}=U_{t\leftarrow s}^{\zeta}(y)$. In this context, the Jacobian of $U$ is defined by
    \[J_{t\leftarrow s}^{\zeta}(y)\cdot e := \frac{d}{d\varepsilon}\bigg|_{\varepsilon=0}U_{t\leftarrow s}^{\zeta}(y+\varepsilon e) .\]
    Then the following facts hold true for $J_{t\leftarrow s}$:
\begin{enumerate}[wide, labelwidth=!, labelindent=0pt, label=\textnormal{(\roman*)}]
\setlength\itemsep{.1in}    

\item Setting $\mathcal{V}_{s,t}^{e}=J_{t\leftarrow s}(y)\cdot e$, then $\mathcal{V}$ solves a RDE on $[s,T]$. Namely, for $t\in [s,T]$ we have 
    \begin{equation}\label{eq: Jacobian-RDE}
        \mathcal{V}_{s,t}^{e}=e + \int_{s}^{t}\left(\int_{U}Db(r,x_{r},a)\gamma_{r}(da) \right) \mathcal{V}_{s,r}^{e}\,dr +\int_{s}^{t}D\sigma (r,x_{r})\mathcal{V}_{s,r}^{e}\, d\zeta_{r}.
    \end{equation}

  \item The following estimate hold:
    \begin{equation} \label{ineq: Jacob}
        \|J_{\cdot \leftarrow s}(y)\|_{\infty;[s,T]}\leq C\exp{(CT\|\zeta\|_{\alpha}^{1/\al})} . \end{equation}
      \end{enumerate}
\end{proposition}
\begin{proof}
   The proof of (i) is a straightforward elaboration of \cite[Chapter 11]{friz-victoir}. In particular, item (ii) is proved along the same lines as \cite[Proposition 11.13]{friz-victoir}. 
   \end{proof}

We are now ready to state a lemma about the behavior of the dynamics according to the size of $\beta$ in the spike variation.
\begin{lemma} \label{Lemma: X}
We work under the same assumptions as in Proposition \ref{prop: Flow and Jacobian}. 
For $\beta>0$ let $x^{\beta}$ be the solution to~(\ref{eq:relaxed-rde}), considered with $s=0$ and the relaxed control $\gamma^{\beta}$ in (\ref{spike}). Also, recall from Definition~\ref{def:spike variation} that $(\overline{x},\overline{\gamma})$ is an optimal pair. Then the function $\beta\mapsto x^{\beta}$ is right differentiable, where $x^{\beta}$ has to be understood as an element of $\mathcal{Q}_{\zeta}^{\kappa}(\R^n)$. Moreover, the right derivative $V^{\beta}:= \partial_{\beta}x^{\beta}$ satisfies the following linear rough differential equation: For $t<t_{0}+\beta$ we have $V^{\beta}_{t}=0$ and for $t\geq t_{0}+\beta$ the path $V^{\beta}$ solves 
\begin{equation} \label{eq: RDE for derivative}
V^{\beta}_{t}=g^{\beta}+\int_{t_{0}+\beta}^{t}\lp\int_{U}Db(r,x_{r}^{\beta},a)\ga_{r}^{\beta}(da)\rp V_{r}^{\beta}dr +\int_{t_{0}+\beta}^{t}D\sigma(r,x_{r}^{\beta})V_{r}^{\beta}d\zeta_{r},
\end{equation}
where the quantity $g^{\beta}$ is given by
\begin{equation} \label{eq: g-beta}
g^{\beta}=\int_{U}b(t_{0}+\beta,x^{\beta}_{t_{0}+\beta},a)\mu(da)-\int_{U}b(t_{0}+\beta,x^{\beta}_{t_{0}+\beta},a)\overline{\ga}_{t_{0}+\beta}(da). 
\end{equation}
In addition, the following estimate holds true:
\begin{equation} \label{Ineq: derivative_of_dynamics}
\mathcal{N}\left[V^{\beta};\mathcal{Q}_{\zeta}^{\ka}([0,T];\R^m)\right]\leq C\exp\lp CT\|\zeta\|_{\alpha}^{1/\alpha} \rp .
\end{equation}
In particular, one has
\begin{equation}\|x^{\beta}-\overline{x}\|_{\infty;[0,T]}=O(\beta) .
\label{eq: Pontry-1}\end{equation}
        \end{lemma}
        
\begin{proof}
For a generic $\beta>0$, consider the rough differential equation
\begin{equation} \label{eq: dynamics-beta-0}
    x_{t}^{\beta}=y+\int_{0}^{t}\lp \int_{U}b(r,x_{r}^{\beta},a)\ga_{r}^{\beta}(da)\rp\,dr+\int_{0}^{t}\sigma(r,x_{r}^{\beta})\,d \zeta_{r}.
\end{equation}
Then, by Proposition \ref{ExistUniq2} there exists a unique solution $x^{\beta}$ for every given initial value $y$. So for $t\geq t_{0}$, we can consider the flow map $U_{t\leftarrow t_{0}}(\overline{x}_{t_{0}})=x_{t}^{\beta}$ (Here we use the fact that $x^{\beta}\equiv \overline{x}$ on $[0,t_{0}]$). By Proposition \ref{prop: Flow and Jacobian} we obtain that $\mathcal{V}^{e}_{t_{0}t}=J_{t\leftarrow t_{0}}(\overline{x}_{t_{0}+\beta})\cdot e$ satisfies
\[\mathcal{V}_{t_{0},t}^{e}=e+\int_{t_{0}}^{t}\lp \int_{U}Db(r,x_{r}^{\beta},a)\gamma_{r}^{\beta}(da) \rp \mathcal{V}_{t_{0},r}^{e}\,dr +\int_{t_{0}}^{t}D\sigma(r,x_{r}^{\beta})\mathcal{V}_{t_{0},r}^{e}\,d\zeta_{r}.\]
We now wish to examine the differentiability of the application $\beta\mapsto x^{\beta}$. We will divide this analysis in several steps.

\smallskip
\noindent
\textit{Step 1: Reduction of the problem.}
One way to prove the differentiability of $\beta\mapsto x^{\beta}$ is to proceed by Picard iterations. Namely, we set $x_{t}^{\beta,0}=y$. Then, we iteratively define $x^{\beta,k+1}=\Gamma(x^{\beta,k})$, with $\Gamma$ given as an application
\begin{equation*}
  \Gamma\colon \mathcal{Q}_{\zeta}^{\ka}(\R^{m})  \longrightarrow \mathcal{Q}_{\zeta}^{\ka}(\R^m) ,
\quad\text{such that}\quad
  w  \longmapsto \Gamma(w) = z \, ,
\end{equation*}
 where $z$ is defined by
 \[z_{t}=y+\int_{0}^{t}\lp \int_{U}b(r,w_{r},a)\ga_{r}^{\beta}(da)\rp\,dr+\int_{0}^{t}\sigma(r,w_{r})\,d \zeta_{r}.\]
 With this notation in hand, a possible strategy to get differentiability results in \cite[Chapter 11]{friz-victoir}  is the following:
   \begin{enumerate}[wide, labelwidth=!, labelindent=0pt, label=\textnormal{(\arabic*)}]
   \setlength\itemsep{.05in}
       \item It is well known that the Picard scheme $x^{\beta,k}$ is convergent in $\mathcal{Q}_{\zeta}^{\kappa}(\R^m)$. It converges to the solution $x^{\beta}$ of (\ref{eq: dynamics-beta-0}).
       
       \item One can also prove differentiability by induction. Namely, if one assumes that $\beta \mapsto x^{\beta,k}$ is differentiable, the problem boils down to establish differentiability of
       \begin{equation}\label{c1}
 \beta\in [0,T]  \longmapsto \Gamma(x^{\beta,k})\in\mathcal{Q}_{\zeta}^{\ka}(\R^m).
\end{equation}
In particular, it will be easy to show that $x^{\beta,k}|_{[0,t_{0}+\beta]}=x^{\beta+h,k}|_{[0,t_{0}+\beta]}$.

\item\label{it:differentiability-3} 
Once we have obtained differentiability in \eqref{c1}, together with the uniform (in $k$) bounds for the derivatives, some compactness arguments allow to take limits and prove that $x^{\beta}$ in ~(\ref{eq: dynamics-beta-0}) is differentiable.
\end{enumerate}
For sake of conciseness, we will refrain from implementing the whole strategy summarized above. We will content ourselves with giving some details about differentiability in the main step~\eqref{c1}. Moreover, we shall only focus on right derivatives in $\beta$, which are sufficient for our purposes.

Summarizing our considerations so far, our differentiation problem is now reduced to the following: Assume $\beta \mapsto x^{\beta}$ is right differentiable at $\beta$. Namely, we assume that the following limit exists in $\mathcal{Q}_{\zeta}^{\kappa}$:
\begin{equation}\label{eq: Derivative of xbeta}
    X^{\beta}:=\lim_{h\downarrow 0}\frac{x^{\beta+h}-x^{\beta}}{h}.
\end{equation}
 Next, set $z^{\beta}=\Gamma(x^{\beta})$. We wish to prove that $z^{\beta}$ is also differentiable as an element of $\mathcal{Q}_{\zeta}^{\ka}$. In addition, we want to derive an equation like (\ref{eq: RDE for derivative}) and estimates such as (\ref{Ineq: derivative_of_dynamics}) for $Z^{\beta}\equiv \partial_{\beta}z^{\beta}$. The remainder of the proof will focus on this step.

\smallskip
\noindent
\textit{Step 2: Time decomposition of the dynamics.} 
In order to get an expression for $Z^{\beta}$ defined as above, we will write the equation for $z^{\beta}$ in a more enlightening way (according to the decomposition of $\gamma^{\beta}$). Specifically, we divide the interval $[0,T]$ in three  pieces:

\begin{enumerate}[wide, labelwidth=!, labelindent=0pt, label=\textnormal{(\alph*)}]
\setlength\itemsep{.05in} 

\item \label{case: t<t_{0}}
For $t<t_{0}$, the variation on $\overline{\ga}$ has not occurred yet. Specifically, going back to (\ref{spike}), we have $\gamma^{\beta}_{t}=\overline{\gamma}_{t}$ regardless of $\beta$. Hence we have 
\begin{equation}\label{eq: dynamics-equiv}
z_{t}^{\beta}={u}_{t}, 
\quad \text{for all} \quad 
t\in [0,t_{0}],
\end{equation}
where $u_{t}=\Gamma(x^{\beta}_{t})$ is independent of $\beta$.

\item
Next, for $t\in[t_{0},t_{0}+\beta]$, the dynamics (\ref{eq: dynamics-beta-0}) can be read as
\begin{equation}\label{eq:decomp-dynamics}
    z_{t}^{\beta}={u}_{t_{0}}+\int_{t_{0}}^{t}\lp\int_{U}b(r,x_{r}^{\beta},a)\mu(da) \rp\,dr +\int_{t_{0}}^{t}\sigma(r,x_{r}^{\beta})\,d\zeta_{r} .
\end{equation}
Moreover, for $h>0$ the path $y^{\beta+h}$ also follows the dynamics (\ref{eq:decomp-dynamics}) on $[t_{0},t_{0}+\beta]$.

\item
Eventually, for $t\geq t_{0}+\beta$, equation (\ref{eq: dynamics-beta-0}) becomes
\begin{multline}\label{eq:decomp-dynamics-2}
z_{t}^{\beta}={u}_{t_{0}}+\int_{t_{0}}^{t_{0}+\beta}\lp \int_{U}b(r,x_{r}^{\beta},a)\mu(da) \rp\,dr \\
+\int_{t_{0}+\beta}^{t}\lp \int_{U}b(r,x_{r}^{\beta},a)  \overline{\ga}_{r}(da) \rp \,dr 
+ \int_{t_{0}}^{t}\sigma(r,x_{r}^{\beta})\,d\zeta_{r}.
\end{multline}
\end{enumerate}
We are now ready to derive an expression for $Z^{\beta}=\partial_{\beta}z^{\beta}$, following the decompositions $(\ref{eq: dynamics-equiv})$-$(\ref{eq:decomp-dynamics})$-$(\ref{eq:decomp-dynamics-2})$ . We deal with the previous three cases separately.

\begin{enumerate}[wide, labelwidth=!, labelindent=0pt, label=\emph{(\roman*)}]
\setlength\itemsep{.1in}     

\item \emph{Case $t\in [0,t_{0})$.}
    Earlier, we saw in (\ref{eq: dynamics-equiv}) that $z^{\beta}_{t}={u}_{t}$. Similarly, we have $z^{\beta +h}_{t}={u}_{t}$. Therefore, we conclude that $Z^{\beta}_{t}\equiv 0$.
    
    \item \emph{Case $t\in [t_{0},t_{0}+\beta]$.}
Looking at (\ref{eq:decomp-dynamics}) and observing that $z^{\beta+h}_{t}$ also satisfies (\ref{eq:decomp-dynamics}), we conclude that $z^{\beta}\equiv z^{\beta+h}$. Ultimately, we yet again have $Z^{\beta}_{t}=0$.
    
    \item \emph{Case $t\in (t_{0}+\beta,T]$.}
   Invoking the dynamics of $x^{\beta}$ and $x^{\beta+h}$, and using the observation that $x^{\beta+h}_{t_{0}+\beta}=x^{\beta}_{t_{0}+\beta}$, we obtain:
   \begin{multline*}
       \frac{z_{t}^{\beta+h}-z^{\beta}_{t}}{h}
       =\frac{1}{h}\int_{t_{0}+\beta}^{t}\int_{U}b(r,x_{r}^{\beta+h},a)\ga_{r}^{\beta+h}(da)\,dr 
       -\frac{1}{h}\int_{t_{0}+\beta}^{t}\int_Ub(r,x_{r}^{\beta},a)\ga_{r}^{\beta}(da)\,dr\\
       +\frac{1}{h}\int_{t_{0}+\beta}^{t}\{\sigma(x_{r}^{\
       \beta+h})-\sigma(r,x_{r}^{\beta})\}\,d\zeta_{r}.
\end{multline*}
Therefore we get the following decomposition for the increments of $z^{\beta}$:
\begin{equation}
    \frac{z^{\beta+h}_{t}-z^{\beta}_{t}}{h}=I_{h}+II_{h}+III_{h},
\label{eq: decomp of y}
\end{equation}
where the quantity $I_{h}$ is given by
\begin{equation}\label{eq: Ih}
 I_{h} 
 =\frac{1}{h}\int_{t_{0}+\beta}^{t}\int_{U}\{b(r,x_{r}^{\beta+h},a)-b(r,x_{r}^{\beta},a)\}\ga_{r}^{\beta}(da)\,dr ,
\end{equation}
and where $II_{h}$ and $III_{h}$ are respectively defined by
\begin{align}
  \label{eq: IIh} II_{h}&=\frac{1}{h}\int_{t_{0}+\beta}^{t}\int_{U}b(r,x_{r}^{\beta+h},a)\{\ga^{\beta+h}_{r}-\ga^{\beta}_{r}\}(da)\, dr\\
     \label{eq: IIIh}  III_{h}&=\frac{1}{h} \int_{t_{0}+\beta}^{t}\{\sigma(r,x_{r}^{\beta+h})-\sigma(r,x_{r}^{\beta})\}\,d\zeta_{r}
\end{align}
We shall derive the limits in those three terms separately.

\smallskip
\noindent
\textit{Step 3: Limit for $I_{h}$.} 
Let us deal with the term $I_{h}$ in (\ref{eq: Ih}) first. This is the easiest of the terms in~\eqref{eq: decomp of y}. Indeed, since we have assumed $x^{\beta}$ is differentiable in~(\ref{eq: Derivative of xbeta}), one can easily check that
\begin{equation}
\lim_{h\downarrow0}I_{h}=\int_{t_{0}+\beta}^{t}\lp\int_{U}Db(r,x_{r}^{\beta},a)\ga_{r}^{\beta}(da) \rp X_{r}^{\beta}\,dr.
\label{eq: limit Ih}\end{equation}

\smallskip
\noindent
\textit{Step 4: Limit for $III_{h}$.} 
We now handle the term $III_{h}$, that we wish to differentiate inside the rough integral in~\eqref{eq: IIIh}. We claim that 
\begin{equation}\label{eq: Integral_3}
    \lim_{h\to 0}III_{h}=\int_{t_{0}+\beta}^{t} D\sigma(r,x_{r}^{\beta})X_{r}^{\beta}\,d\zeta_{r} .
\end{equation}
In order to prove (\ref{eq: Integral_3}), let us give a more explicit expression for the rough integral. Namely (see \cite{gubinelli} for details) for $x^{\beta}$, the solution to (\ref{eq: dynamics-beta-0}), the strongly controlled decomposition for the rough integral can be written
\begin{equation}\label{eq: Rough_Exp-Int}
    \int_{u}^{v}\sigma(r,x_{r}^{\beta})\, d\zeta_{r}=\sigma(u,x_{u}^{\beta})\delta \zeta_{uv}+ D\sigma(u,x_{u}^{\beta})\sigma(u,x_{u}^{\beta})\zeta_{uv}^{2}+ \Lambda(\ell^{\beta})_{uv},
\end{equation}
where $\Lambda$ is the sewing map introduced in Proposition \ref{prop:La} and $\ell^{\beta}$ is an increment in $\mathcal{C}_{3}^{3\ka}$ defined by
\begin{equation}\label{eq: ell-beta}
    \ell^{\beta}=\delta\lp \sigma(\cdot,x^{\beta})\delta \zeta +D\sigma(\cdot,x^{\beta})\sigma(\cdot,x^{\beta})\zeta^{2}\rp.
\end{equation}
Now if $x^{\beta}$ is differentiable in $\beta$, it is clear that the first two terms in the right hand side of~(\ref{eq: Rough_Exp-Int}) are also differentiable. The term $\ell^{\beta}$ as given by (\ref{eq: ell-beta}) is also clearly differentiable. In order to see that $\Lambda(\ell^{\beta})$ is differentiable, we use 
(\ref{eq: ell-beta}) in Proposition \ref{prop:La} to assert
\begin{align*}
   \left\| \Lambda(\ell^{\beta+h})- \Lambda(\ell^{\beta})-h\Lambda(\partial_{\beta}\ell^{\beta})\right\|_{3\kappa}
   &=
   \left\| \Lambda(\ell^{\beta + h} -\ell^{\beta}-h\partial_{\beta}\ell^{\beta}) \right\|_{3\kappa} \\ 
   &\lesssim 
   \left\|\ell^{\beta+h}-\ell^{\beta}-h\partial_{\beta}\ell^{\beta}\right\|_{3\kappa}=o(h) \, ,
\end{align*}
where the last identity stems from the fact that $\ell^{\beta}$ is differentiable with respect to $\beta$. This argument implies that $\partial_{\beta}\Lambda(\ell^{\beta})=\Lambda(\partial_{\beta}\ell^{\beta})$. Plugging this identity back into (\ref{eq: Rough_Exp-Int}), we have proved that (\ref{eq: Integral_3}) holds true.

\smallskip
\noindent
\textit{Step 5: Limit for $II_{h}$.} 
Let us now turn to the analysis of the term $II_{h}$ in (\ref{eq: IIh}). Owing to the expressions $\ga^{\beta}$ and $\ga^{\beta+h}$, it is readily checked that 
\[II_{h}=\frac{1}{h}\int_{t_{0}+\beta}^{(t_{0}+\beta+h)\wedge t}\int_{U}b(r,x_{r}^{\beta+h},a)\{\mu-\overline{\ga}_{r}\}(da)\,dr.\]
We will further decompose this expression as 
\begin{equation}\label{eq:IIh limit}
    II_{h}=II_{h}^{(1)}+II_{h}^{(2)},
\end{equation}
where $II_{h}^{(1)}$, $II_{h}^{(2)}$ are respectively defined by
\begin{align}
    II_{h}^{(1)}&=\frac{1}{h}\int_{t_{0}+\beta}^{(t_{0}+\beta+h)\wedge t}\int_{U}\{b(r,x_{r}^{\beta+h},a)-b(r,x_{r}^{\beta},a)\}\{\mu-\overline{\ga}_{r}\}(da)\,dr\\
    II_{h}^{(2)}&=\frac{1}{h}\int_{t_{0}+\beta}^{(t_{0}+\beta+h)\wedge t}\int_{U}b(r,x_{r}^{\beta},a)\{\mu-\overline{\ga}_{r}\}(da)\,dr.
\end{align}
Moreover, if $\beta\mapsto x^{\beta}$ is differentiable and $b$ is smooth, we have
\[|b(r,x_{r}^{\beta+h},a)-b(r,x_{r}^{\beta},a)|\lesssim|h|,\]
uniformly in $(r,a)$. Therefore, combining this regularity with the regularity given by the time integral in $II_{h}^{(1)}$, we trivially  get
\begin{equation}\label{eq: II-1h limit}
    \lim_{h\downarrow 0}II_{h}^{(1)}=0.
\end{equation}
As far as $II_{h}^{(2)}$ is concerned, one can apply the fundamental theorem of calculus and obtain
\begin{equation}\label{eq: II-2h limit}
    \lim_{h\downarrow 0}II_{h}^{(2)}= \int_{U}b(t_{0}+\beta,x^{\beta}_{t_{0}+\beta},a)\mu(da)-\int_{U}b(t_{0}+\beta,x^{\beta}_{t_{0}+\beta},a)\overline{\ga}_{t_{0}+\beta}(da).
\end{equation}
Gathering (\ref{eq: II-1h limit}) and (\ref{eq: II-2h limit}) into (\ref{eq:IIh limit}), we end up with
\begin{equation}\label{eq: limit IIh}
    \lim_{h\downarrow 0} II_{h}=\int_{U}b(t_{0}+\beta,x^{\beta}_{t_{0}+\beta},a)\mu(da)-\int_{U}b(t_{0}+\beta,x^{\beta}_{t_{0}+\beta},a)\overline{\ga}_{t_{0}+\beta}(da).
\end{equation}

\smallskip
\noindent
\textit{Step 6: Proof of differentiability.} \label{lem 3.8: step 6}
We can now summarize and obtain an expression for $Z^{\beta}=\partial_{\beta}z^{\beta}$. Namely putting together (\ref{eq: limit Ih}), (\ref{eq: Integral_3}) and (\ref{eq: limit IIh}), and recalling our decomposition (\ref{eq: decomp of y}), we obtain that $Z^{\beta}$ satisfies

\begin{multline}\label{eq: Equation-Derivative 1}
Z^{\beta}_{t}=\int_{t_{0}+\beta}^{t}\lp\int_{U}Db(r,x_{r}^{\beta},a)\ga_{r}^{\beta}(da) \rp X_{r}^{\beta}\,dr + \int_{t_{0}+\beta}^{t}D\sigma(r,x_{r}^{\beta})X_{r}^{\beta}\,d\zeta_{r}\\
    \hspace{0.25cm}+ \int_{U}b(t_{0}+\beta,x^{\beta}_{t_{0}+\beta},a)\mu(da)-\int_{U}b(t_{0}+\beta,x^{\beta}_{t_{0}+\beta},a)\overline{\ga}_{t_{0}+\beta}(da).
\end{multline}
Here again, it should be mentioned that the limits in $h$ have to be understood in the rough paths sense.
\end{enumerate}

Recall that in point \ref{it:differentiability-3} of our strategy for differentiation, we took limits in Picard iterations related to~(\ref{eq: Equation-Derivative 1}). Skipping the tedious details for the justification of this step, we obtain that if $x^{\beta}$ is the solution of (\ref{eq: dynamics-beta-0}), and if we set $V^{\beta}=\partial_{\beta}x^{\beta}$, then $V^{\beta}$ solves the following linear rough differential equation on $[t_{0}+\beta,T]$:
\begin{multline}\label{eq: Equation-Derivative 2}
V^{\beta}_{t}=\int_{t_{0}+\beta}^{t}\lp\int_{U}Db(r,x_{r}^{\beta},a)\ga_{r}^{\beta}(da) \rp V_{r}^{\beta}\,dr + \int_{t_{0}+\beta}^{t}D\sigma(r,x_{r}^{\beta})V_{r}^{\beta}\,d\zeta_{r}\\
    \hspace{0.25cm}+ \int_{U}b(t_{0}+\beta,x^{\beta}_{t_{0}+\beta},a)\mu(da)-\int_{U}b(t_{0}+\beta,x^{\beta}_{t_{0}+\beta},a)\overline{\ga}_{t_{0}+\beta}(da).
\end{multline}
We have thus proved relation (\ref{eq: RDE for derivative}).

\smallskip
\noindent
\textit{Step 7: Bounds on the derivative.} 
We now turn to the proof of inequality (\ref{Ineq: derivative_of_dynamics}). Here note that we are looking for a bound for the whole norm $\mathcal{N}\left[ V^{\beta};\mathcal{Q}_{\zeta}^{\kappa} \right]$, as defined in (\ref{eq:norm-controlled}). For the sake of conciseness, we will just bound $\mathcal{N}[V^{\beta}; \mathcal{C}_{1}^{\infty}(\R^m)]$, leaving the other lengthy computations to the reader. Now in order to estimate the sup norm of $V^{\beta}$, we resort to the fact that $V^{\beta}\equiv 0$ on $[0,t_{0}+\beta)$. Furthermore, on $[t_{0}+\beta,T]$, the path $V^{\beta}$ satisfies the same equation as $J^{\zeta}_{t\leftarrow t_{0}+\beta}(y)\cdot \hat{e}$, where $J$ stands for the Jacobian in Proposition \ref{prop: Flow and Jacobian} and $\hat{e}$ is given by
\begin{equation}
   \hat{e}= \int_{U}b(t_{0}+\beta,x^{\beta}_{t_{0}+\beta},a)\mu(da)-\int_{U}b(t_{0}+\beta,x^{\beta}_{t_{0}+\beta},a)\ga_{t_{0}+\beta}(da) .
\end{equation}
Therefore, invoking the estimate \eqref{ineq: Jacob} on the Jacobian and the fact that $b$ (and thus $\hat{e}$) is uniformly bounded, we obtain
\begin{multline*}
   \mathcal{N}[V^{\beta};\mathcal{C}_{1}^{0}([0,T])]=\sup_{t\in[t_{0}+\beta,T]}|V^{\beta}_{t}|
    =\sup_{t\in[t_{0}+\beta,T]}|J_{t\leftarrow t_{0}+\beta}^{\zeta}(y)\cdot \hat{e}|\\
   \leq C_{b}\sup_{t\in[t_{0}+\beta,t]}\|J_{t\leftarrow t_{0}+\beta}(x^{\beta}_{t_{0}+\beta})\|_{\mathrm{op}}
    \leq C\exp\lp CT\|\zeta\|_{\alpha}^{1/\alpha} \rp,
    \end{multline*} 
    which proves
    \begin{equation}\|V^{\beta}\|_{\infty;[0,T]}\leq C\exp \lp C T \|\zeta\|_{\alpha}^{1/\alpha} \rp.
\label{eq: uniform-estimate}    \end{equation}
Note again that we have focused on the uniform norm in (\ref{eq: uniform-estimate}), letting the reader complete the other estimates leading to (\ref{Ineq: derivative_of_dynamics}).

Eventually, observing that $x^{0}\equiv \overline{x}$, we get that 
\begin{equation}\label{eq: FTC x}x^{\beta}_{t}-\overline{x}_{t}=\int_{0}^{\beta}V_{t}^{u} \, du.\end{equation}
Hence, relation (\ref{eq: uniform-estimate}) directly yields (\ref{eq: Pontry-1}). This finishes our proof.\end{proof}

\subsection{Identification of a derivative process}\label{sec:identification-derivative}
In this section 
we first introduce the so called variational equation, which will aid us in convenient cancellations later on to obtain the Pontryagin maximal principle. We will then study some basic properties of the solution $Y^{\beta}$ to the variational equation and obtain a Taylor type formula. Those considerations are still preliminary steps towards our Pontryagin principle.

\begin{definition} \label{def: variational equation}
 Suppose that the assumptions of Theorem \ref{thm:Optimal Pair} prevail. Let $(\overline{x},\overline{\gamma})$ be an optimal pair according to Theorem \ref{thm:Optimal Pair}. Also, let $I$ be the interval and $\mu$ the measure which were both defined in Definition \ref{def:spike variation}. We define a path $Y^{\beta}\colon[0,T]\to\R^m$ by $Y_{t}^{\beta}=0$ for $t<t_{0}$, and for $t\geq t_{0}$, $Y^{\beta}$ solves the rough differential equation
  \begin{equation}\label{eq: Y}
Y_{t}^{\beta}=L_{t}^{\beta}+\int_{t_0}^{t}\lp\int_{U}Db(r,\overline{x}_{r},a)\overline{\ga}_{r}(da) \rp Y_{r}^{\beta}dr + \int_{t_{0}}^{t}D\sigma(r,\overline{x}_{r})Y_{r}^{\beta}\,d\zeta_{r},
 \end{equation}
 where $L^{\beta}$ is defined by
 \[L_{t}^{\beta}=\int_{t_{0}}^{t} \left\{ \int_{U}b(r,\overline{x}_{r},a)\mu(da)-\int_{U}b(r,\overline{x}_{r},a)\overline{\gamma}_{r}(da) \right\}\mathbf{1}_{I}(r)dr.\]
 Equation (\ref{eq: Y}) is called the variational equation along the optimal pair.
\end{definition}

Next we compute another derivative, namely we differentiate the process $Y^{\beta}$ introduced in Definition \ref{def: variational equation}.

\begin{lemma}\label{Lemma: Y} We work under the same assumptions of Definition \ref{def: variational equation} and Proposition \ref{prop: Flow and Jacobian}. For $\beta>0$, let $Y^{\beta}$ be the unique solution of the rough differential equation (\ref{eq: Y}). Also, recall from Definition \ref{def:spike variation} that $(\overline{x},\overline{\ga})$ is an optimal pair. Then $\beta \mapsto Y^{\beta}$ is right differentiable, where $Y^{\beta}$ has to be understood as an element of $\mathcal{Q}_{\zeta}^{\ka}(\R^m)$. Moreover, the right derivative $W^{\beta}:=\partial_{\beta}Y^{\beta}$ satisfies the following linear rough differential equation: For $t<t_{0}+\beta$ we have $W_{t}^{\beta}=0$, and for $t\geq t_{0}+\beta$ the path $W^{\beta}$ solves
\begin{equation}\label{eq: dynamics_W}
    W^{\beta}_{t}=h^{\beta}+\int_{t_{0}+\beta}^{t}\lp\int_{U}Db(r,\overline{x}_{r},a)\overline{\ga}_{r}(da) \rp W_{r}^{\beta}\,dr +\int_{t_{0}+\beta}^{t}D\sigma(r,\overline{x}_{r})W^{\beta}_{r}\, d\zeta_{r},
\end{equation}
where the quantity $h^{\beta}$ is given by
\begin{equation}
    \label{eq: h_beta}h^{\beta}=\int_{U}b(t_{0}+\beta,\overline{x}_{t_{0}+\beta},a)\mu(da)-\int_{U}b(t_{0}+\beta,\overline{x}_{t_{0}+\beta},a)\overline{\ga}_{t_{0}+\beta}(da).
\end{equation}
In addition, the following estimate holds true:
\begin{equation}\label{eq: estimate_W}
    \mathcal{N}[W^{\beta};\mathcal{Q}_{{\zeta}}^{\zeta}([0,T];\R^m)]\leq C\exp\lp C \|\zeta\|_{\alpha}^{1/\alpha} \rp.
\end{equation}
In particular, one has
\begin{equation}\label{eq: Pontry-2}
\|Y^{\beta}\|_{\infty,[0,T]}=O(\beta). 
\end{equation}
\end{lemma}

\begin{proof}
For a generic $\beta>0$, recall that $Y^{\beta}$ solves equation (\ref{eq: Y}). In particular $Y^{\beta}_{t}=0$ for $t\leq t_{0}$. In order to compare $Y^{\beta+h}$ with $Y^{\beta}$, we will have to consider the flows from $t_{0}+\beta$  to $t$. Namely we define the flow map $Y_{t}^{\beta}=U_{t\leftarrow t_{0}+\beta}^{\zeta}(Y_{t_{0}+\beta}^{\beta})$. By Proposition \ref{prop: Flow and Jacobian} we obtain that $\mathcal{V}_{t_{0}+\beta,t}^{e}=J_{t\leftarrow t_{0}+\beta}^{\zeta}(Y^{\beta}_{t_{0}+\beta})\cdot e$ satisfies
   \begin{equation}\label{eq: Jacobian for Y}\mathcal{V}_{t_{0}+\beta,t}^{e}=e+\int_{t_{0}+\beta}^{t} \lp \int_{U}Db(r,\overline{x}_{r},a)\overline{\ga}_{r}(da)\rp \mathcal{V}_{t_{0}+\beta,r}^{e}\,dr +\int_{t_{0}+\beta}^{t}D\sigma(r,\overline{x}_{r})\mathcal{V}_{t_{0}+\beta,r}^{e}\,d\zeta_{r}. \end{equation}
Proceeding in the same manner as in Lemma  \ref{Lemma: X}, we study the right differentiability of $\beta \mapsto Y^{\beta}$. In order to avoid unnecessary repetitions, we will skip some of the details of computation.

\smallskip
\noindent
\textit{Step 1: Reduction of the problem.} Similarly to Lemma \ref{Lemma: X}, we define
\begin{equation*}
    \Gamma \colon \mathcal{Q}_{\zeta}^{\ka}(\R^m) \longrightarrow \mathcal{Q}_{\zeta}^{\ka}(\R^m), \quad \text{such that}\quad w\longmapsto \Gamma(w)=z,
\end{equation*}
where $z$ is defined by
\begin{multline*}
    z_{t}=\int_{0}^{t}\lp \int_{U}Db(r,\overline{x}_{r},a)\overline{\ga}_{r}(da)\rp w_{r}\,dr +\int_{0}^{t}D\sigma(r,\overline{x}_{r},a)w_{r}\,d\zeta_{r}\\+
    \int_{0}^{t}\left\{\int_{U}b(r,\overline{x}_{r},a)\ga_{r}^{\beta}(da)-\int_{U}b(r,\overline{x}_{r},a)\overline{\ga}_{r}(da) \right\}\,dr.
\end{multline*}
Note that with respect to (\ref{eq:decomp-dynamics-2}), we removed the $\mathbf{1}_{I}(r)$ and encapsulated it in $\ga^{\beta}_{r}$.
We follow the same strategy stated in Lemma \ref{Lemma: X}. Summarizing our considerations, our differentiation problem is now reduced to the following: Assume $\beta \mapsto w^{\beta}$ is right differentiable at $\beta$. Namely, we assume that the following limit exists in $\mathcal{Q}_{\zeta}^{\ka}$:
\begin{equation}
    \tilde{W}^{\beta}=\lim_{h\downarrow 0}\frac{w^{\beta+h}-w^{\beta}}{h}.
\end{equation}
Next, we set $z^{\beta}=\Gamma(w^{\beta})$. We wish to prove that $z^{\beta}$ is also differentiable as an element of $\mathcal{Q}_{\zeta}^{\ka}$. In addition, we want to derive an equation like (\ref{eq: dynamics_W}) and an estimate such as (\ref{eq: estimate_W}) for $Z^{\beta}\equiv \partial_{\beta}z^{\beta}$. The remainder of the proof will focus on this step.

\smallskip
\noindent
\textit{Step 2: Time decomposition of the dynamics.} Exactly like in the proof of Lemma \ref{Lemma: X}-Step 2, it is beneficial to decompose the dynamics of $z^{\beta}$ in different time intervals. However, it can be shown like in (\ref{eq:decomp-dynamics})-(\ref{eq:decomp-dynamics-2}) that the dynamics on $[0,t_{0}+\beta]$ gives rise to a trivial derivative. We will thus focus our attention on the case $t\in(t_{0}+\beta,T]$. 
 Namely, for $t\in (t_{0}+\beta,T]$, we notice that $z^{\beta}$ can be expressed as
\begin{multline}\label{eq: dynamics_z_3}
    z_{t}^{\beta}=\int_{0}^{t}\lp \int_{U}Db(r,\overline{x}_{r},a)\overline{\ga}_{r}(da)\rp w_{r}^{\beta}\,dr +\int_{0}^{t}D\sigma(r,\overline{x}_{r})w_{r}^{\beta}\,d\zeta_{r}\\ +
    \int_{t_{0}}^{t_{0}+\beta}\left\{\int_{U}b(r,\overline{x}_{r},a)\mu(da)-\int_{U}b(r,\overline{x}_{r},a)\overline{\ga}_{r}(da) \right\}\,dr.
\end{multline}
We are now ready to derive an expression for $Z^{\beta}=\partial_{\beta}z_{\beta}$ on $(t_{0}+\beta,T]$, following the decomposition (\ref{eq: dynamics_z_3}). 
\smallskip
\noindent
 That is, invoking the dynamics of $z^{\beta}$ and $z^{\beta+h}$, and using the observation that $z^{\beta+h}_{t_{0}+\beta}=z_{t_{0}+\beta}^{\beta}$, we obtain:
\begin{multline}
    \frac{z^{\beta+h}_{t}-z^{\beta}_{t}}{h}=\int_{t_{0}+\beta}^{t}\lp\int_{U}Db(r,\overline{x}_{r},a)\overline{\ga}_{r}(da) \rp \frac{w_{r}^{\beta +h}-w_{r}^{\beta}}{h}dr +\int_{t_{0}+\beta}^{t}D\sigma(r,\overline{x}_{r})\frac{w_{r}^{\beta+h}-w_{r}^{\beta}}{h}d\zeta_{r}\\
    +\frac{1}{h}\int_{t_{0}+\beta}^{t_{0}+\beta+h}\left\{\int_{U}b(r,\overline{x}_{r},a)\mu(da)-\int_{U}b(r,\overline{x}_{r},a)\overline{\ga}_{r}(da) \right\}dr.
\end{multline}
Arguing in a similar fashion as Step 6 of case \ref{lem 3.8: step 6} in Lemma \ref{Lemma: X}, we can take limits as $h\to 0$ and we obtain an equation for $Z^{\beta}$, which satisfies:

\begin{multline}\label{eq: Zbeta}
    Z^{\beta}_{t}=\int_{t_{0}+\beta}^{t}\lp\int_{U}Db(r,\overline{x}_{r},a)\overline{\ga}_{r}(da) \rp \tilde{W}^{\beta}_{r}\,dr +\int_{t_{0}+\beta}^{t}D\sigma(r,\overline{x}_{r})\tilde{W}^{\beta}_{r}\, d\zeta_{r}\\
+\int_{U}b(t_{0}+\beta,\overline{x}_{t_{0}+\beta},a)\mu(da)-\int_{U}b(t_{0}+\beta,\overline{x}_{t_{0}+\beta},a)\overline{\ga}_{t_{0}+\beta}(da).
\end{multline}
Hence, starting from $(\ref{eq: Zbeta})$ and mimicking the argument of Step 6 from Lemma \ref{Lemma: X}, we obtain that if $Y^{\beta}$ is a solution of (\ref{eq: Y}). Moreover if we set $W^{\beta}=\partial_{\beta}Y^{\beta}$, then $W^{\beta}$ solves the following linear rough differential equation
\begin{multline}
W^{\beta}_{t}=\int_{t_{0}+\beta}^{t}\lp\int_{U}Db(r,\overline{x}_{r},a)\overline{\ga}_{r}(da) \rp W_{r}^{\beta}\,dr +\int_{t_{0}+\beta}^{t}D\sigma(r,\overline{x}_{r})W^{\beta}_{r}\, d\zeta_{r}\\
+\int_{U}b(t_{0}+\beta,\overline{x}_{t_{0}+\beta},a)\mu(da)-\int_{U}b(t_{0}+\beta,\overline{x}_{t_{0}+\beta},a)\overline{\ga}_{t_{0}+\beta}(da).
\end{multline}
We have thus proved that $(\ref{eq: dynamics_W})$ holds true.

\smallskip

\noindent
\textit{Step 3: Bounds on Derivatives.}
  We now turn to the proof of inequality (\ref{eq: estimate_W}). Here we should note that we are looking for the bound for the whole norm $\mathcal{N}[W^{\beta}; \mathcal{Q}_{\zeta}^{\ka}]$, as defined in (\ref{eq:norm-controlled}). As in the proof of Lemma \ref{Lemma: X}, we will just bound $\mathcal{N}[W^{\beta};\mathcal{C}_{1}^{\infty}]$. Now in order to estimate the sup norm of $W^{\beta}$, we resort to the fact that $W^{\beta}\equiv 0$ on $[0,t_{0}+\beta]$. Furthermore, on $[t_{0}+\beta,T]$, the path $W^{\beta}$ satisfies the same equation as (\ref{eq: Jacobian for Y}), where $e$ is to be taken as $h^{\beta}$ from (\ref{eq: h_beta}).
    Therefore, invoking estimate (\ref{ineq: Jacob}) on the Jacobian and the fact that $b$ (and thus $e$) is bounded, we obtain the same estimate as (\ref{eq: uniform-estimate}):
\begin{equation}\label{estimate: W}\mathcal{N}[W^{\beta},\mathcal{C}_{1}^{0}([0,T])]\leq C\mathrm{exp}(CT\|\zeta\|_{\alpha}^{1/\alpha}).\end{equation}
Note again that the details allowing to go from (\ref{estimate: W}) to (\ref{eq: estimate_W}) are left to the reader. Eventually, observing that $Y^{0}=0$, the equivalent of (\ref{eq: FTC x}) reads
\[Y^{\beta}_{t}=\int_{0}^{\beta}W^{u}_{t}\,du.\]
Hence, relation (\ref{eq: estimate_W}) directly yields (\ref{eq: Pontry-2}). This finishes our proof.\end{proof}
On the road to prove the Pontryagin maximum principle we need a proposition stating that the path $Y^{\beta}$ is an approximate derivative of $\beta \mapsto x^{\beta}$ (see forthcoming Lemma \ref{prop: Prop Z}). Before we dive into that, 
let us recall a basic definition of distance for probability measures which will be used in the sequel.
\begin{definition}\label{def: FM}
    Let $X$ be a metric space. Let $\mu$ and $\nu$ be two probability measures on $X$. 
    We define the \textit{Fortet-Mourier distance} between these probability measures as
    \begin{equation}\label{eq: FM}
        \mathrm{d_{FM}}(\mu,\nu)=\sup_{f\in \mathcal{M}}\left|\int_{X} f d\mu -\int_{X}fd\nu \right|,
    \end{equation}
    where $\mathcal{M}=\{f:X\to \R: f \text{ is bounded and Lipschitz}\}$.
\end{definition}
\begin{remark}
    \label{rmk: FM} We have the following trivial bound, which will also be referred to in the sequel:
    \begin{equation}\label{est: FM}
        \mathrm{d_{FM}}(\mu,\nu)\leq 2.
    \end{equation}
\end{remark}

As another preliminary step towards identification of $Y^{\beta}$ as an approximate derivative,
we will state variants of known stability results and estimates. The first one will be a H\"older-norm variant of \cite[Proposition 2.6]{allan-cohen}, whose proof is included here in order to exhibit the dependence in $\beta$. For simplicity of exposition, we establish the next result for $\sigma$ depending only on its spatial variable.

\begin{proposition}\label{prop: variant-AC}
   Let $\zeta$ be a rough path satisfying Hypothesis \ref{hyp:zeta}. Let $b\in C_{b}^{0,1,1}([0,T]\times \R^m \times U;\R^m)$ and $\sigma\in C^{3}_{b}(\R^{m};\R^{m,d})$. Let $\eta,\nu\colon[0,T]\to\mathcal{P}(U)$ be two relaxed controls. Also, we consider the solutions $x^{\eta}$ and $x^{\nu}$ of the rough differential equations (\ref{eq:relaxed-rde}) (with $s=0$), each corresponding to the relaxed controls $\eta$ and $\nu$ respectively. Recall from Definition \ref{def:weakly-ctrld} that the Gubinelli derivative of a controlled process $z$ is denoted by $z^{\zeta}$ and its remainder by $\rho^{z}$. Then we have
   \begin{equation}\label{eq: est-beta}
       \|(x^{\eta})^{\zeta}-({x}^{\nu})^{\zeta}\|_{\alpha;[0,T]}+\|\rho^{x^{\eta}}-\rho^{{x}^{\nu}}\|_{2\alpha;[0,T]}\lesssim \lp \int_{0}^{T}\mathrm{d_{FM}}(\eta_{r},\nu_{r})^{\frac{1}{1-2\alpha}}\,dr \rp^{1-2\alpha},
   \end{equation}
   where $\mathrm{d_{FM}}(\cdot,\cdot)$ is the Fortet-Mourier distance from Definition \ref{def: FM}.
\end{proposition}
\begin{proof}
 A controlled process decomposition (see Definition \ref{def:weakly-ctrld}) for $x^{\eta}$ and ${x}^{\nu}$ is easily obtained. Namely, if $x^{\eta}$ is defined by (\ref{eq:relaxed-rde}) with $\ga=\eta$, set 
\[\mathcal{I}_{st}^{\eta}=\int_{s}^{t}\sigma(x_{r}^{\eta}) \, d\zeta_{r}.\]
Then we have
\[ (x_{s}^{\eta})^{\zeta}=\sigma(x_{s}^{\eta}), \quad \text{and} \quad \rho^{x^{\eta}}_{st}=\rho^{\mathcal{I}^{\eta}}_{st}+\int_{s}^{t}\int_{U}b(r,x_{r}^{\eta},a)\eta(da)dr.\]
The same type of decomposition holds true for ${x}^{\nu}$. Therefore we get
\begin{equation} \label{eq: est-eq: ABC}
    E_{T}:=\|(x^{\eta})^{\zeta}-({x}^{\eta})^{\zeta}\|_{\alpha;[0,T]}+\|\rho^{x^{\eta}}-\rho^{{x}^{\nu}}\|_{2\alpha;[0,T]}\lesssim \|A\|_{\alpha} + \|B\|_{2\alpha}+\|C\|_{2\alpha} ,
\end{equation}
where $A\in \mathcal{C}_{1}^{\alpha}$ and $B, C \in \mathcal{C}_{2}^{2\alpha}$ are respectively defined as
\begin{align}
        \label{eq: A}A_{t}&:=\sigma(x^{\eta}_{t})-\sigma({x}^{\nu}_{t}),\\
        \label{eq: B}B_{st}&:=\int_{s}^{t}\int_{U}b(r,x_{r}^{\eta},a)\eta(da)dr-\int_{s}^{t}\int_{U}b(r,{x}^{\nu}_{r},a){\nu}_{r}(da)dr,\\
        \label{eq: C}C_{st}&:=\rho^{\mathcal{I}^{\eta}}_{st}-\rho^{{\mathcal{I}}^{\nu}}_{st}.
    \end{align}
    Let us examine the drift term $B$ above. It is readily checked from (\ref{eq: B}) that 
    \[|B_{st}|\leq \left|\int_{s}^{t}\int_{U}\{b(r,x_{r}^{\eta},a)-b(r,{x}_{r}^{\nu},a)\}\eta_{r}(da)dr \right| + \left|\int_{s}^{t}\int_{U}b(r,{x}_{r}^{\nu},a)\{\eta_{r}(da)-{\nu}_{r}(da)\}dr \right|. \]
    We now specialize our setting to the case $\alpha \in (1/3,1/2]$, the easy modification for the case $\alpha>1/2$ being left to the patient reader. Under those assumptions, owing to the regularity conditions on $b$ we get
    \begin{align*}
    |B_{st}|&\leq C_{b}\|x^{\eta}-x^{\nu}\|_{\infty}|t-s|+ C_{b} \int_{0}^{T}\mathbf{1}_{[s,t]}(r)\mathrm{d}_{FM}(\eta_{r},\nu_{r})\,dr\\
    & \leq  C_{b}\|x^{\eta}-x^{\nu}\|_{\alpha}T^{\alpha}|t-s|^{2\alpha}T^{1-2\alpha}+C_{b}|t-s|^{2\alpha}\lp \int_{0}^{T} \mathrm{d}_{FM}(\eta_{r},\nu_{r})^{\frac{1}{1-2\alpha}}\,dr \rp^{1-2\alpha},
        \end{align*}
        where we have invoked H\"older's inequality and the definition of Fortet-Mourier distance from Definition \ref{def: FM} for the second upper bound.
This ultimately implies that 
\begin{equation} \label{eq: II-1}\|B\|_{2\alpha}\leq C_{b}T^{1-\alpha}\|x^{\eta}-{x}^{\nu}\|_{\alpha}+C_{b}\lp \int_{0}^{T}\mathrm{d_{FM}}(\eta_{r},\nu_{r})^{\frac{1}{1-2\alpha}} \rp^{1-2\alpha}.
\end{equation}
Moreover, note that
\begin{align*}|\delta x^{\eta}_{st}- \delta {x}^{\nu}_{st}|&=|\{ (x_{s}^{\eta})^{\zeta}-({x}^{\nu}_{s})^{\zeta}\} \delta \zeta_{st} + \{\rho^{x^{\eta}}_{st}-\rho^{{x}^{\nu}}_{st}\}|\\
&\leq \|(x^{\eta})^{\zeta}-({x}^{\nu})^{\zeta}\|_{\alpha}\|\zeta\|_{\alpha}|t-s|^{\alpha} +\|\rho^{x^{\eta}}-\rho^{{x}^{\nu}}\|_{2\alpha}|t-s|^{2\alpha}, \end{align*}
which implies that
\begin{equation}\label{eq: II-2}
\|x^{\eta}-{x}^{\nu}\|_{\alpha}
\leq 
T^{\alpha}\|\zeta\|_{\alpha}\|(x^{\eta})^{\zeta}-({x}^{\nu})^{\zeta}\|_{\alpha}
+ T^{\alpha}\|\rho^{x^{\eta}}-\rho^{{x}^{\nu}}\|_{2\alpha}.\end{equation}
Substituting (\ref{eq: II-2}) into (\ref{eq: II-1}), we obtain
\begin{equation}\label{eq: II}
    \|B\|_{2\alpha} 
    \leq 
    C_{b}T\lp \|\zeta\|_{\alpha}\|(x^{\eta})^{\zeta}-(x^{\nu})^{\zeta}\|_{\alpha} 
    + \|\rho^{x^{\eta}}-\rho^{x^{\nu}}\|_{2\alpha}\rp 
    +C_{b}\lp\int_{0}^{T} \mathrm{d_{FM}}(\eta_{r},\nu_{r})^{\frac{1}{1-2\alpha}} \rp^{1-2\alpha}.
\end{equation}
The terms $\|A\|_{\alpha}$ and $\|C\|_{2\alpha}$ in (\ref{eq: est-eq: ABC}) are handled similarly to e.g \cite[Chapter 11]{friz-victoir},  thanks to some classical rough paths considerations. We let the reader check that
\begin{align}\label{eq: I}
    \|A\|_{\alpha}&\lesssim \|x^{\eta}-{x}^{\nu}\|_{\alpha}\|\zeta\|_{\alpha}+ \|\rho^{x^{\eta}}-\rho^{{x}^{\nu}}\|_{2\alpha}T^{\alpha},\\
\label{eq: III}
    \|C\|_{2\alpha}&\lesssim \lp\|(x^{\mu})^{\zeta}-({x}^{\nu})^{\zeta}\|_{\alpha}+\|\rho^{x^{\mu}}-\rho^{{x}^{\nu}}\|_{2\alpha} \rp \|\zeta\|_{\alpha}.
\end{align}
Combining (\ref{eq: II})-(\ref{eq: I})-(\ref{eq: III}) in (\ref{eq: est-eq: ABC}), we end up with 
\begin{equation} \label{eq: E}
    E_{T}\leq C E_{T} +C_{b}\lp\int_{0}^{T}\mathrm{d_{FM}}(\eta_{r},\nu_{r})^{\frac{1}{1-2\alpha}}\,dr \rp^{2\alpha},
\end{equation}
where $C=C(T,\|\zeta\|_{\alpha;[0,T]})$ increases in time. Choosing $\tilde{T}$ small enough, from (\ref{eq: E}) we obtain
\begin{equation}
    E_{t}\leq \lp \int_{0}^{t}\mathrm{d_{FM}}(\eta_{r},\nu_{r})^{\frac{1}{1-2\alpha}}dr \rp ^{2\alpha}, \quad \text{for all}\quad t\leq \tilde{T}.
\end{equation}
Applying standard pasting arguments, we can exhaust the interval $[0,T]$ to obtain (\ref{eq: est-beta}).
\end{proof}
The next proposition compares the derivative $V^{\beta}=\partial_{\beta}x^{\beta}$ with the derivative $W^{\beta}=\partial_{\beta}Y^{\beta}$. Its proof should be seen as a variant of \cite[Lemma 3.21]{Lew}.
\begin{proposition}\label{prop: variant-TL}
    Let $\zeta$ be a rough path satisfying Hypothesis \ref{hyp:zeta}. Let $b\in C_{b}^{0,1,1}([0,T]\times \R^m \times U;\R^m)$ and $\sigma\in C^{3}_{b}(\R^{m};\R^{m,d})$. Consider an optimal control $\overline{\ga}$ from Proposition \ref{thm:Optimal Pair} and its spike variation $\ga^{\beta}$ defined as in (\ref{spike}). Let $V^{\beta}$ and $W^{\beta}$ be the solutions of the rough differential equations (\ref{eq: RDE for derivative}) and (\ref{eq: dynamics_W}) respectively. Then
    \begin{equation}
      \label{eq: main-estimate}  \|V^{\beta}-W^{\beta}\|_{\infty;[0,T]}\lesssim |g^{\beta}-h^{\beta}|+ \beta^{1-2\alpha}.
    \end{equation}
\end{proposition}
\begin{proof}
    Recall that $V^{\beta}$ and $W^{\beta}$ solve (\ref{eq: RDE for derivative}) and (\ref{eq: dynamics_W}) respectively. 
    Set $Z^{\beta}_{t}:=W^{\beta}_{t}-V^{\beta}_{t}$.
    Using our adjusted Proposition~\ref{prop: variant-AC} and applying a similar argument as in \cite[Lemma 3.21]{Lew}, we obtain
\begin{equation}\label{eq: post-gron}
    \|Z^{\beta}\|_{\infty;[0,T]}\lesssim |g^{\beta}-h^{\beta}| + \lp \int_{0}^{T}\mathrm{d_{FM}}(\ga_{r}^{\beta},\overline{\ga}_{r})^{\frac{1}{1-2\alpha}}\,dr \rp^{1-2\alpha}.
\end{equation}
Using the fact that $\ga^{\beta}\equiv \overline{\ga}$ on $[0,t_{0})\cup (t_{0}+\beta,T]$ and implementing  the trivial bound~(\ref{est: FM}) for the Fortet-Mourier distance, relation~(\ref{eq: post-gron})
gives us
 \begin{equation}
    \|Z^{\beta}\|_{\infty;[0,T]}\lesssim |g^{\beta}-h^{\beta}|+\beta^{1-2\alpha},
\end{equation}
which is exactly (\ref{eq: main-estimate}).
\end{proof}
We are now ready to state the proposition asserting that $Y^{\beta}$ is an approximate derivative of the perturbation $x^{\beta}$ of $\bar{x}$.
\begin{proposition}\label{prop: Prop Z}
Working under the same assumptions as in Lemma \ref{Lemma: X} and Lemma \ref{Lemma: Y} we have
    \begin{equation}\label{Pontry-3}\|x^{\beta}-\overline{x}-Y^{\beta}\|_{\infty,[0,T]}=o(\beta).\end{equation}
\end{proposition}

\begin{proof}
In Lemma \ref{Lemma: X} and Lemma \ref{Lemma: Y}, we found rough differential equations satisfied by the derivatives of $x^{\beta}$ and $Y^{\beta}$ with respect to $\beta$, which we named $V^{\beta}$ and $W^{\beta}$ respectively (see (\ref{eq: RDE for derivative}) and (\ref{eq: dynamics_W})). Using the fact that $x^{0}\equiv 0$, $Y^{0}\equiv 0$, and the fundamental theorem of calculus, we obtain the following identity:
\begin{equation}
   \label{eq: obs} x^{\beta}_{t}-\overline{x}_{t}-Y^{\beta}_{t}=\int_{0}^{\beta}\{V^{u}_{t}-W^{u}_{t}\}\,du=-\int_{0}^{\beta}Z_{t}^{u}\,du,
\end{equation}
where the path $Z^{\beta}=W^{\beta}-V^{\beta}$ has been introduced in the proof of Proposition \ref{prop: variant-TL}.
In order to obtain our desired outcome (\ref{Pontry-3}), the previous observation (\ref{eq: obs}) implies that it only suffices to show 
\begin{equation} \label{eq: Step 1}
|Z^{\beta}_{t}|=|V^{\beta}_{t}-W^{\beta}_{t}|\lesssim \beta^{\varepsilon},
\end{equation}
for some $\varepsilon>0$. In addition, Proposition \ref{prop: variant-TL} entails that (\ref{eq: Step 1}) holds true as long as we prove that for some $\varepsilon'>0$,
\begin{equation}\label{eq: g-h-est}
    |g^{\beta}-h^{\beta}|\lesssim \beta^{\varepsilon'}.
\end{equation}
Now recall the expression (\ref{eq: g-beta}) for $g^{\beta}$ and (\ref{eq: h_beta}) for $h^{\beta}$. A close examination of those expressions shows that (\ref{eq: g-h-est}) is an easy consequence of (\ref{eq: Pontry-1}) alongside the Lipschitz assumption on $b$. This finishes the proof.\end{proof}
Our final lemma towards Pontryagin's principle is a comparison between the rewards related to $\ga$ and $\ga^{\beta}$.

\begin{lemma}\label{lem: Taylor J}Let $\overline{\ga}$ be the optimal control from Theorem \ref{thm:Optimal Pair}. Consider the spike variation $\ga^{\beta}$ from Definition \ref{def:spike variation}. We assume that $F\in C_{b}^{0,1,0}([0,T]\times\R^m\times \mathcal{P}(U);\R)$ and $G\in C_{b}^{1}(\R^{m};\R)$. Let $J_{T}(\cdot)$ be the reward function (\ref{eq:J}). Then we have the following Taylor-like expansions of the reward functional with respect to the spike variation of the control variable:
    \begin{multline} \label{eq: Taylor J} J_{T}(\gamma^{\beta},y)-J_{T}(\overline{\gamma},y)=\left<DG(\overline{x}_{T}),Y_{T}^{\beta}) \right>+\int_{0}^{T} \left<DF(r,\overline{x}_{r},\overline{\gamma}_{r}),Y_{r}^{\beta}\right>\, dr \\ \hspace{3.5cm} + \int_{0}^{T}\{F(r,\overline{x}_{r},\mu)-F(r,\overline{x}_{r},\overline{\gamma}_{r})\}\mathbf{1}_{I}(r)\,dr +o(\beta),
    \end{multline}
    where $Y^{\beta}$ is defined by (\ref{eq: Y}) and the interval $I$ is introduced in Definition \ref{def:spike variation}.
\end{lemma}

\begin{proof} 
Write the definition (\ref{eq:J}) of $J_{T}$ for both $\ga^{\beta}$ and $\overline{\ga}$, then differentiate $F$ and $G$ along an interpolation from $\overline{x}$ to $x^{\beta}$. We let the reader check that we obtain
    \begin{multline}
       J_{T}(\gamma^{\beta},y)-J_{T}(\overline{\gamma},y)=\int_{0}^{T}\int_{0}^{1}\left<DF(t,\overline{x}_{t}+\theta (x^{\beta}_{t}-\overline{x}_{t}),\gamma^{\beta}_{t}),x^{\beta}_{t}-\overline{x}_{t} \right>\, d\theta \,dt \\
       +\int_{0}^{T}\{F(t,\overline{x}_{t},\gamma_{t}^{\beta})-F(t,\overline{x}_{t},\overline{\gamma}_{t})\}\, dt 
+\int_{0}^{1}\left<DG(\overline{x}_{T}+\theta (x^{\beta}_{T}-\overline{x}_{T})), x_{T }^{\beta}-\overline{x}_{T} \right>d\theta.
\end{multline}
Next we add and subtract terms of the form $DF(t,\overline{x}_{t},\ga^{\beta})$ and $DG(\overline{x}_{T})$. We end up with
\begin{multline*}
     J_{T}(\ga^{\beta},y)-J_{T}(\overline{\ga},y)  = \int_{0}^{T}\int_{0}^{1}\left<DF(t,\overline{x}_{t}+\theta (x^{\beta}_{t}-\overline{x}_{t}),\gamma^{\beta}_{t})-DF(t,\overline{x}_{t},\gamma^{\beta}_{t}),x^{\beta}_{t}-\overline{x}_{t} \right>\, d\theta \,dt \\
    +\int_{0}^{1}\left< DG(\overline{x}_{T}
       +\theta (x^{\beta}_{T}-\overline{x}_{T}))-DG(\overline{x}_{T}),x^{\beta}_{T}-\overline{x}_{T} \right> d\theta
+\int_{0}^{T}\left<DF(t,\overline{x}_{t},\gamma^{\beta}_{t}),x^{\beta}_{t}-\overline{x}_{t} \right>\, dt \\
       + \left<DG(\overline{x}_{T}),x^{\beta}_{T}-\overline{x}_{T}) \right>+\int_{0}^{T}\{F(t,\overline{x}_{t},\mu)-F(t,\overline{x}_{t},\overline{\gamma}_{t})\}1_{I}(t)\,dt .
       \end{multline*}
      After reordering some of the terms above to match with (\ref{eq: Taylor J}) and some algebraic manipulation, we obtain:
   \begin{multline}\label{eq: TaylorJ1}
J_{T}(\ga^{\beta},y)-J_{T}(\overline{\ga},y)=\left<DG(\overline{x}_{T}),Y^{\beta}_{T}\right> +\int_{0}^{T}\left<DF(t,\overline{x}_{t},\overline{\gamma}_{t}),Y_{t}^{\beta}
        \right>\,dt\\ 
        + \int_{0}^{T}\{F(t,\overline{x}_{t},\mu)-F(t,\overline{x}_{t},\overline{\gamma}_{t})\}1_{I}(t)\,dt+ A+B+C+D, 
   \end{multline}
   where the remainders $A,B,C,D$ are respectively defined by
   \begin{align*}
A&=\left<DG(\overline{x}_{T}),x^{\beta}_{T}-\overline{x}_{T}-Y^{\beta}_{T} \right>,\\
       B&=\int_{0}^{1}\left<DG(\overline{x}_{T}+\theta (x^{\beta}_{T}-\overline{x}_{T}))-DG(\overline{x}_{T}),x^{\beta}_{T}-\overline{x}_{T} \right>\, d\theta,\\
       C&=\int_{0}^{T}\left< DF(t,\overline{x}_{t},\overline{\gamma}_{t}),x^{\beta}_{t}-\overline{x}_{t}-Y^{\beta}_{t})\right>dt,\\
       D&=\int_{0}^{T}\int_{0}^{1}\left<DF(t,\overline{x}_{t}+\theta (x^{\beta}_{t}-\overline{x}_{t}),\gamma^{\beta}_{t})-DF(t,\overline{x}_{t},\gamma^{\beta}_{t}),x^{\beta}_{t}-\overline{x}_{t} \right>\, d\theta \,dt.
   \end{align*}
   We can now invoke our preliminary lemmas in order to upper bound the remainder terms.
   Specifically since $DG$ is bounded, by applying Lemma \ref{prop: Prop Z} we have $A=o(\beta)$. Also, using the Lipschitz property of $DG$ and applying Lemma \ref{Lemma: X} twice, we obtain $B=O(\beta^{2})$. Moreover, observing that $DF$ is bounded and applying Lemma \ref{prop: Prop Z} we get $C=o(\beta)$. Finally, employing our Lipschitz assumption on $DF$ and applying Lemma \ref{Lemma: X} twice, we establish $D=O(\beta^2)$. Putting things together, by (\ref{eq: TaylorJ1}) and the fact that $A+B+C+D=o(\beta)$, we obtain relation~(\ref{eq: Taylor J}).
 \end{proof}

\subsection{Proof of Theorem \ref{Pontryagin}}
        After the long series of preliminary results in Sections~\ref{sec: Prelim} and~\ref{sec:identification-derivative}, we now turn to the proof of our main theorem for ththe Pontryagin maximum principle.
        
\begin{proof}[Proof of Theorem \ref{Pontryagin}]
With our preliminary bounds in hand, the proof is similar to the deterministic case~\cite{yong-zhou} as well as the rough case with non-relaxed controls \cite{diehl}. We might thus skip some of the details for conciseness. First, notice that
\begin{equation}\label{eq: FTC-pY}
\left<DG(\bar{x}_{T}),Y_{T}^{\beta} \right>=\left<p_{T},Y_{T}^{\beta} \right> -\left<p_{0},Y_{0}^{\beta} \right>,
\end{equation}
where we have used the fact that $Y_{0}^{\beta}=0$ (see Definition \ref{def: variational equation}) and the terminal condition for $p_{\cdot}$ from (\ref{RDEp}).
Next, since $\zeta$ is a geometric rough path, applying the rough product rule (see \cite[Theorem 7.7]{MR4174393}) to the right hand side of (\ref{eq: FTC-pY}) we obtain
\begin{equation}\label{eq: RoughPR}
    \left< DG(\overline{x}_{T}),Y_{T}^{\beta} \right> =\int_{0}^{T}d\left<p_{r},Y_{r}^{\beta} \right>
    =\int_{0}^{T}\left<p_{r},dY_{r}^{\beta} \right>+\int_{0}^{T}\left<dp_{r},Y_{r}^{\beta} \right>.
\end{equation}
Finally, plugging in the expressions of $dp_{t}$ and $dY_{t}^{\beta}$ (from (\ref{RDEp}) and (\ref{eq: Y}) respectively) into~(\ref{eq: RoughPR}), we observe that the rough terms cancel out. This yields
\begin{multline}\label{eq: RoughPR1}
    \left<DG(\overline{x}_{T}),Y_{T}^{\beta} \right>=\int_{0}^{T}\left <p_{r},\left\{\int_{U}b(r,\overline{x}_{r},a)\mu(da)-\int_{U}b(r,\overline{x}_{r},a)\overline{\gamma}_{r}(da) \right\}\mathbf{1}_{I}(r) \right>dr\\-\int_{0}^{T}\left<Y_{r}^{\beta},DF(r,\overline{x}_{r},\overline{\gamma}_{r}) \right>dr .
\end{multline}
Now recall that we have obtained an expression for $J_{T}(\ga^{\beta},y)-J_{T}(\overline{\ga},y)$ in (\ref{eq: Taylor J}). Since $\overline{\ga}$ is an optimal control, we thus get
\begin{align*}
0 \geq J_{T}(\gamma^{\beta},y)-J_{T}({\overline\gamma},y)&=\left<DG(\overline{x}_{T}),Y_{T}^{\beta} \right>+\int_{0}^{T} \left<DF(r,\overline{x}_{r},\overline{\gamma}_{r}),Y_{r}^{\beta}\right>\, dr \\ & \quad+ \int_{0}^{T}\{F(r,\overline{x}_{r},\mu)-F(r,\overline{x}_{r},\overline{\gamma}_{r})\}\mathbf{1}_{I}(r)\,dr +o(\beta).
\end{align*}
Hence substituting (\ref{eq: RoughPR1}) into the above identity, some simple algebraic manipulations yield
\begin{align*}
  J_{T}(\ga^{\beta},y)-J_{T}(\overline{\ga},y)  &= \int_{t_{0}}^{t_{0}+\beta}\left <p_{r},\int_{U}b(r,\overline{x}_{r},a)\mu(da)-\int_{U}b(r,\overline{x}_{r},a)\overline{\gamma}_{r}(da)\right>dr\\
    &\quad + \int_{t_{0}}^{t_{0}+\beta}\{F(r,\overline{x}_{r},\mu)-F(r,\overline{x}_{r},\overline{\gamma}_{r})\}dr +o(\beta)
\end{align*}
Finally, dividing by $\beta$ and letting $\beta \to 0$ will give us the desired relation~(\ref{PMP}).\end{proof}

\section{$Q$-Learning}\label{sec:Q}
Recall that the reward function on an interval $[t_{0},T]$ is given by $J_{t_{0},T}$ in \eqref{eq:J}. In order to optimize the policy $\ga$ in real time, it is natural  to quantify the variations of $J_{t_0,T}$ when one modifies $\ga$ on small intervals. This type of method is commonly labeled as $Q$-learning in the reinforcement learning literature, and we will develop those variational tools in the current section. Note that the small interval perturbation technique is reminiscent of what we have seen in Section~\ref{sec: Pontry}, where differentiations along small interval variations have been used to derive Pontryagin's maximal principle. Invoking those techniques, in the current section we will introduce a notion of $Q$-function and derive several possible definitions of $q$-functions.

\subsection{Definition of the $Q$-function}
Let us start from a path $x^{t_0,y,\ga}$ as given in \eqref{eq:relaxed-rde}. The main type of variation we shall consider consists in changing the {path measure} $\ga$ for a constant measure $\mu$ on an interval of the form $[t_0, t_0+\beta]$, similarly to what we had done in Definition~\ref{def:spike variation}. Let us label some notation before turning to the corresponding computations.

\begin{definition}\label{def:f1}
Consider a small parameter $\beta>0$, a strategy $\ga \in \mathcal{V}^{\varepsilon,L}$(see (\ref{eq:cv^ep,L}) for the definition), a fixed $t_0\in[0,T]$ and a fixed relaxed control $\mu\in \mathcal{P}(U)$. Then
\begin{enumerate}[wide, labelwidth=!, labelindent=0pt, label=\emph{(\arabic*)}]
\setlength\itemsep{.1in}         
\item Similarly to Definition \ref{def:spike variation}, we define a spike variation $\hat{\ga}^{\beta}$ of $\ga$ as 
    \begin{equation}\label{eq: Qspike} \hat{\ga}^{\beta}_{t}=
        \begin{cases}
            \mu\, , \quad \text{if} \quad t\in I\\
            \ga_{t}\, , \quad \text{if} \quad t\notin I,
        \end{cases}
    \end{equation}
    where the interval $I$ is $I_{\beta}=[t_0,t_0+\beta]$.
    \item We call $\hat{x}^{\beta}$ the solution of (\ref{eq:relaxed-rde}) starting at $y\in \R^{n}$ at time $t_0$ and controlled by $\hat{\ga}^{\beta}$. Whenever needed we will highlight the dependence on all parameters as $\hat{x}^{t_0,y,\beta,\mu,\ga}$.
\end{enumerate}
\end{definition}
\begin{remark}
    One can decompose the dynamics $\hat{x}^{\beta}$ exactly as in (\ref{eq:decomp-dynamics})-(\ref{eq:decomp-dynamics-2}). We obtain
    \begin{equation}\label{eq:rde-a-ga}
        \begin{aligned}
            \hat{x}_{t}^{\beta}&=y+\int_{t_0}^{t}\int_U b(r,x_{r}^{{\mu}},u)\mu(du)dr+\int_{t_0}^{t}\sigma(r,x_{r}^{{\mu}})d\zeta_{r}, \quad t\in[t_0,t_0+\beta)\\
            \hat{x}_{t}^{\beta}&=x_{t_0+\beta}^{\mu}+\int_{t_0+\beta}^{t}\int_{U}b(r,{\hat{x}_{r}^{\beta}},u)\ga_{r}(du)dr+\int_{t_0+\beta}^{t}\sigma(r,{\hat{x}_{r}^{\beta}})d\zeta_{r},\quad t\in[t_0+\beta,T],
        \end{aligned}
    \end{equation}
where $x^{\mu}$ is the solution of (\ref{eq:relaxed-rde}) starting at $y$ at time $t_{0}$ and controlled by $\mu$.
\end{remark}

With Definition~\ref{def:f1} in hand, the $Q$-function alluded to above is the reward related to $\hat{x}^{\beta}$. We proceed to define this function in our context.
\begin{definition}\label{def: perturbed-dynamics}Fix $t_0\in [0,T]$, $y\in \R^{m}$, $\beta>0$, $\ga \in \mathcal{V}^{\varepsilon,L}$ and {$\mu \in \mathcal{P}(U)$}. Let $\hat{x}^{\beta}$ be the path solving the rough differential equation (\ref{eq:rde-a-ga}). Also, we consider two functions $F$ and $G$ which satisfy Hypothesis \ref{hyp:reward}. We define the reward-type function $Q_{\beta}$ in the following way:
\beq\label{eq:Q}
Q_{\beta}(t_0,y,\mu;\ga) = \int_{t_0}^{t_0+\beta} F(s, \hat{x}^{\beta}_{s}, \mu) ds + \int_{t_0+\beta}^T F(s, \hat{x}^{\beta}_{s}, \ga_s) ds + G(\hat{x}^{\beta}_{T}).
\eeq
\end{definition}

\begin{remark}\label{rem: zero-perturb}
    Recall the reward function $J_{t,T}$ from equation (\ref{eq:J}). One notices that, when the perturbation parameter $\beta$ is equal to 0, we have
    \begin{equation}
        Q_{0}(t_0,y,\mu;\ga)=J_{t_0,T}(y;\ga).
    \end{equation}
Moreover, recalling (\ref{eq: Qspike}) and (\ref{eq:rde-a-ga}), it is readily seen that $Q_{\beta}$ is a reward-type function. Indeed, we have the following identity, whose easy proof is left to the reader:
    \begin{equation}
        Q_{\beta}(t_0,y,\mu;\ga)=J_{t_0,T}(y;\hat{\ga}^{\beta}).
    \end{equation}
    \end{remark}
    
\subsection{Defining the $\mathbf{q}$-function through differentiation}\label{sec:q-function}
Our first definition of $q$-function will follow the steps of Section \ref{sec: Pontry}. More specifically, one can differentiate the solution $\hat{x}^{\beta}$ of~(\ref{eq:rde-a-ga}) with respect to the parameter $\beta$ and observe the effects on $Q_{\beta}$. This is reminiscent of what we did in Section \ref{sec: Prelim}. We begin by defining the main derivative of interest in this context.

\begin{definition}\label{def: hat-deriv}
    Let $\hat{x}^{\beta}$ be the solution of the rough differential equation (\ref{eq:rde-a-ga}). We define $\widehat
{V}^{\beta}$ as $\partial_{\beta}\hat{x}^{\beta}$, similarly to what we did in Lemma \ref{Lemma: X} (see equation (\ref{eq: RDE for derivative})).
\end{definition}

\begin{remark}
    Existence of the derivative process $\hat{V}^{\beta}$ is in fact a small variation on Lemma~\ref{Lemma: X}, replacing $\overline{\ga}$ therein by the path $\ga$ in (\ref{eq: Qspike}). We let the reader check the details.
\end{remark}

As mentioned in \cite{jia2023qlearning}, the function $Q_{\beta}$ in~(\ref{eq:Q}) satisfies the relation
\[\lim_{\beta \to 0} \lln Q_{\beta}(t_{0},y,\mu;\ga)-J_{t_{0},T}(y;\ga) \rrn=0.
\]
One can thus try to differentiate $\beta\mapsto Q_{\beta}(t_{0},y,\mu;\ga)$ in order to get an infinitesimal version of the reward function. We will develop these ideas in our setting.

\begin{proposition}\label{prop:q1}
Let the hypothesis from Definition \ref{def: perturbed-dynamics} prevail. Let $Q_{\beta}$ be the reward-type function (\ref{eq:Q}) and $J_{t,T}$ be the quantity given by (\ref{eq:J}). We will also denote $x^{t,y,\ga}$ as $x^{\ga}$, and recall that the derivative $\widehat{V}^{\beta}$ is given in Definition~\ref{def: hat-deriv}. Then we have the following 
\begin{equation}\label{eq:litq1}
    \lim_{\beta \to 0}\frac{Q_{\beta}(t_0,y,\mu;\ga)-J_{t_0,T}(y;\ga)}{\beta}=q(t_{0},y,\mu;\ga),
\end{equation}
where the function $q$ is defined on $[0,T]\times \R^{m}\times \mathcal{P}(U)\times \mathcal{V}^{\varepsilon}$ by
\begin{equation}\label{eq:litq2}
    q(t_{0},y,\mu;\ga)=F(t_0,y,\mu)-F(t_0,y,\ga_{t_{0}})\\+\int_{t_0}^{T}DF(s,x_{s}^{\ga},\ga_{s})\widehat{V}^{0}_{s}ds  +DG(x_{T}^{\ga})\widehat{V}^{0}_{T}.
\end{equation}
\end{proposition}
\begin{notation}\label{not:q-a}
    The function $q$ in (\ref{eq:litq1})-(\ref{eq:litq2}) is usually called (small) $q$-function in the literature (see \cite{jia2023qlearning}). In particular, when $\mu=\delta_{a}$ for a given $a\in U$, we set $q(t,y,\delta_{a};\ga)\equiv q(t,y,a;\ga)$. 
\end{notation}

\begin{proof}[Proof of Proposition \ref{prop:q1}]
First we decompose the renormalized difference in the left hand side of (\ref{eq:litq1}) in the following way:
   \begin{equation}\label{eq: decomp-ABC}
   \frac{Q_{\beta}(t_0,y,\mu;\ga)-J_{t_0,T}(y;\ga)}{\beta}= A^{\beta} +B^{\beta} +C^{\beta},
   \end{equation}
   where the quantities $A^{\beta}$, $B^{\beta}$,$C^{\beta}$ are respectively defined by
   \begin{align}
       A^{\beta}&=\frac{1}{\beta}\int_{t_0}^{t_0+\beta}\{F(s,\hat{x}_{s}^{\beta},\mu)-F(s,x_{s}^{\ga},\ga_{s})\}ds, \label{eq: decomp-A}\\
       B^{\beta}&=\frac{1}{\beta}\int_{t_0+\beta}^{T}\{F(s,\hat{x}_{s}^{\beta},\ga_{s})-F(s,x_{s}^{\ga},\ga_{s})\}ds,\label{eq: decomp-B}\\
       C^{\beta}&=\frac{1}{\beta}\lp G(\hat{x}_{T}^{\beta})-G(x_{T}^{\ga})\rp\label{eq: decomp-C}.
   \end{align}
   We will study the limit of these three terms separately.
       
 \smallskip
\noindent
\textit{Limit of $A^{\beta}$.} We further decompose $A^{\beta}$ into
\begin{equation}
    A^{\beta}=A^{\beta}_{1}+A^{\beta}_{2}, \label{eq: Abeta}
\end{equation}
where the terms $A_{1}^{\beta}$ and $A^{\beta}_{2}$ are given by
\begin{align}
    A^{\beta}_{1}&=\frac{1}{\beta}\int_{t_0}^{t_0+\beta}\{F(s,\hat{x}_{s}^{\beta},\mu)-F(s,x_{s}^{\ga},\mu)\}ds,\label{eq: decompA1}\\
    A^{\beta}_{2}&=\frac{1}{\beta}\int_{t_0}^{t_0+\beta}\{F(s,x_{s}^{\ga},\mu)-F(s,x_{s}^{\ga},\ga_{s})\}ds. \label{eq: decompA2}
\end{align}
In addition, arguing very similarly to what we did in Lemma \ref{Lemma: X} we get that $\|\hat{x}^{\beta}-x^{\ga}\|_{\infty;[0,T]}=O(\beta)$. Plugging this identity into (\ref{eq: decompA1}) we get
\begin{equation}
    \lim_{\beta\to 0}A_{1}^{\beta}= 0 \label{eq: lim-A1}.
\end{equation}
We now turn to the term $A_{2}^{\beta}$ above, which can be treated like in (\ref{eq: II-2h limit}). That is under our continuity assumptions one can apply the fundamental theorem of calculus and get
\begin{equation}
    \lim_{\beta \to 0}A_{2}^{\beta}=F(t_0,y,\mu)-F(t_0,y,\ga_{t_0}).\label{eq: lim-A2}
\end{equation}
Letting $\beta\to 0$ in (\ref{eq: Abeta}) and plugging (\ref{eq: lim-A1})-(\ref{eq: lim-A2}), it is thus readily checked that
\begin{equation}\label{eq: limitA}
    \lim_{\beta \to 0}A^{\beta}=F(t_0,y,\mu)-F(t_0,y,\ga_{t_0}).
\end{equation}
 \smallskip
\noindent
\textit{Limit of $B^{\beta}$.}The term $B^{\beta}$ in (\ref{eq: decomp-B}) is handled very similarly to $A^{\beta}$. First we decompose $B^{\beta}$ into two parts
\begin{equation} \label{eq: Bbeta}
    B^{\beta}=B^{\beta}_{1}+B^{\beta}_{2},
\end{equation}
where we have set
\begin{align}
    B^{\beta}_{1}&=\frac{1}{\beta}\int_{t_0}^{T}\{F(s,\hat{x}_{s}^{\beta},\ga_{s})-F(s,x_{s}^{\ga},\ga_{s})\}ds,, \label{eq: decomp-B1}\\
    B^{\beta}_{2}&= -\frac{1}{\beta}\int_{t_0}^{t_0+\beta}\{F(s,\hat{x}_{s}^{\beta},\ga_{s})-F(s,x_{s}^{\ga},\ga_{s})\}ds \label{eq: decomp-B2}.
\end{align}
Now exactly as in (\ref{eq: lim-A1}), we let the reader check that
\begin{equation}
    \lim_{\beta\to 0}B_{2}^{\beta}=0. \label{eq: lim-B2}
\end{equation}
Furthermore, the term $B_{1}^{\beta}$ can be treated like (\ref{eq: limit Ih}). That is, invoke the fact that the derivative $\widehat{V}^{\beta}=\partial_{\beta}\hat{x}^{\beta}$ in Definition \ref{def: hat-deriv} exists. Then by applying Leibniz' rule we obtain
\begin{equation}\label{eq: lim-B1}\lim_{\beta \to 0}B^{\beta}_{1}=\int_{t_0}^{T}DF(s,x_{s}^{\ga},\ga_{s})\widehat{V}^{0}_{s}ds.\end{equation}
Hence gathering (\ref{eq: lim-B2}) and (\ref{eq: lim-B1}) into (\ref{eq: Bbeta}) we end up with
\begin{equation}\label{eq: limitB}
    \lim_{\beta \to 0}B^{\beta}=\int_{t_0}^{T}DF(s,x_{s}^{\ga},\ga_{s})\widehat{V}_{s}^{0}ds.
\end{equation}
 \smallskip
\noindent
\textit{Limit of $C^{\beta}$.} The limit of $C^{\beta}$ in (\ref{eq: decomp-C}) is obtained along the same lines as $A^{\beta}$ and $B^{\beta}$. We let the patient reader verify that
\begin{equation}\label{eq: limitC}
    \begin{aligned}
        \lim_{\beta \to 0}C^{\beta}=DG(x_{T}^{\ga})\widehat{V}_{T}^{0}.
    \end{aligned}
\end{equation}
 \smallskip
\noindent
\textit{Conclusion.} Gathering (\ref{eq: limitA})-(\ref{eq: limitB})-(\ref{eq: limitC}) into (\ref{eq: decomp-ABC}), we end up with
\begin{multline*}
    \lim_{\beta \to 0}\frac{Q_{\beta}(t_0,y,\mu;\ga)-J_{t_0,T}(y;\ga)}{\beta}=F(t_0,y,\mu)-F(t_0,y,\ga_{t_0})\\+\int_{t_0}^{T}DF(s,x_{s}^{\ga},\ga_{s})\widehat{V}_{s}^{0}ds  +DG(x_{T}^{\ga})\widehat{V}_{T}^{0},
\end{multline*}
which is what we wanted to show.
\end{proof}
Although expression (\ref{eq:litq2}) is useful, it contains terms that involve the derivative $\widehat{V}^{0}$. Those terms are not easy to interpret and uneasy to exploit for optimization purposes. Below, we derive an expression in terms of the Hamiltonian $H$ in (\ref{eq:Hamil}), which is a more physical quantity. Let us first state a lemma which gives a representation of $\widehat{V}^{0}$ involving the Jacobian of our main dynamics.

\begin{lemma} \label{lem: Vhat}
    Recall that the process $\widehat{V}^{0}$ is given in Definition \ref{def: hat-deriv}, and the Jacobian $J$ is introduced in Proposition \ref{prop: Flow and Jacobian}. Also, recall that the family of functions $g^{\beta}$ is given in (\ref{eq: g-beta}), and observe that in our context we have
    \begin{equation}\label{eq:g0}
        g^{0}=\int_{U}b(t_0,x_{t_0}^{
        \ga},a)\mu(da) -\int_{U}b(t_0,x_{t_0}^{\ga},a)\ga_{t_{0}}(da).
    \end{equation}
    Then for $t\in [t_{0},T]$ we have
    \begin{equation} \label{eq: JacobRHS}
        \widehat{V}^{0}_{t}=J_{t\leftarrow t_{0}}^{\zeta,x^{\ga}}(y)\cdot g^{0}.
    \end{equation}
\end{lemma}
\begin{proof}
    According to equation (\ref{eq: Jacobian-RDE}), the path $J_{t\leftarrow t_{0}}^{\zeta}(y)\cdot g^{0}$ in the right hand side of (\ref{eq: JacobRHS}) solves the equation
    \begin{equation*}
        \mathcal{V}_{t}=g^{0} +\int_{t_{0}}^{t}\lp \int_{U}Db(r,x_{r},a)\ga_{r}(da) \rp \mathcal{V}_{r}\,dr 
        + \int_{t_{0}}^{t}D\sigma(r,x_{r})\mathcal{V}_{r}\,d\zeta_{r}.
    \end{equation*}
    This is exactly equation (\ref{eq: RDE for derivative}) for $\widehat{V}^{0}_{t}$, as referred to in Definition \ref{def: hat-deriv}. By uniqueness of the solution of linear rough differential equations, we deduce the desired identity.
\end{proof}
We are now ready to recast the expression for the $q$-function in terms of our Hamiltonian.
\begin{corollary} \label{cor: q1}
    Let the hypothesis from Proposition \ref{prop:q1} prevail, and recall that the Hamiltonian $H$ is spelled out in (\ref{eq:Hamil}). Then the $q$-function introduced by (\ref{eq:litq2}) can also be written as
    \begin{equation}\label{eq:q'}
        q(t,y,\mu;\ga)=H(t,y,\mu,\nabla_{x}J_{t,T}(y,\ga))-H(t,y,\ga_t,\nabla_{x}J_{t,T}(y,\ga)) \, .
    \end{equation}
\end{corollary}
\begin{proof}
    According to expression (\ref{eq:g0}), we have 
    \begin{equation} \label{eq:Jg-identity1}
        \nabla_{x}J_{t,T}(y,\ga)\cdot g^{0}=  \nabla_{x}J_{t,T}(y,\ga)\cdot \left\{\int_{U}b(t,x_{t}^{\ga},u)\mu(du)-\int_{U}b(t,x_{t}^{\ga},u)\ga_{t}(du)\right\} .
    \end{equation}
    Taking identity (\ref{eq:Hamil}) into account, this reads
    \begin{multline}\label{f1}
         \nabla_{x}J_{t,T}(y,\ga)\cdot g^{0} 
         = \nabla_{x}J_{t,T}(y,\ga)\cdot \int_{U}b(t,x_{t}^{\ga},u)\mu(da) \\
         -\lp H(t,y,\ga_t, \nabla_{x}J_{t,T}(y,\ga))-F(t,y,\ga_t) \rp.
    \end{multline}
    Let us now go back to equation (\ref{eq:J}) for $J_{t,T}$ and compute
    \[
    \nabla_{x}J_{t,T}(\ga,y)
    =\int_{t}^{T}DF(s,x_{s}^{\ga},\ga_{s})\, \mathcal{J}_{s\leftarrow t}^{\zeta}(y) \, ds
    +DG(x_{T}^{\ga}) \, \mathcal{J}_{T\leftarrow t}^{\zeta}(y) \, ,
    \]
    which gives us another expression for $\nabla_{x}J_{t,T}(y,\ga) \cdot g^{0}$:
    \begin{equation}\label{eq:Jg-identity2}
    \nabla_{x}J_{t,T}(\ga,y)\cdot g^{0}=\int_{t}^{T}DF(s,x_{s}^{\ga},\ga_{s})(\mathcal{J}_{s\leftarrow t}^{\zeta,x^{\ga}}\cdot g^{0})ds+DG(x_{T}^{\ga})\mathcal{J}_{T\leftarrow t}^{\zeta}\cdot g^{0}.
     \end{equation}
    Owing to identity (\ref{eq: JacobRHS}), this yields
    \[\nabla_{x}J_{t,T}(\ga,y)\cdot g^{0}=\int_{t}^{T}DF(s,x_{s}^{\ga},\ga_{s})\widehat{V}^{0}_{s}ds+DG(x_{T}^{\ga})\widehat{V}^{0}_{T}.\]
    Comparing the above expression with (\ref{eq:litq2}), we end up with
    \begin{equation}\label{f2}
    \nabla_{x}J_{t,T}(\ga,y)\cdot g^{0}=q(t,y,\mu;\ga)-\lp F(t,y,\mu)-F(t,y,\ga_{t}) \rp.
    \end{equation}
    Comparing the right hand sides of~\eqref{f1} and~\eqref{f2}, we have thus obtained
    \begin{equation*}
        q(t,y,\mu;\ga)=F(t,y,\mu)  +\nabla_{x}J_{t,T}(y,\ga)\cdot \int_{U}b(t,x_{t}^{\ga},u)\mu(du)-H(t,y,\ga_{t},\nabla_{x}J_{t,T}(y,\ga)),
    \end{equation*}
    which coincides with (\ref{eq:q'}). This finishes our proof.
\end{proof}
The above expression $(\ref{eq:q'})$ for $q$ is suitable for computational purposes. However, one might also want to mimic the deterministic Brownian setting, for which $q$ was expressed as a time derivative of the reward $J$. Something similar can be achieved in our rough setting, albeit changing the notion of time derivative for generalized time derivatives taken from our strongly controlled process notions of Section \ref{sec:strongProcess}. We summarize our findings in the proposition below.
\begin{proposition} \label{prop:qfunc}
     We still assume that the hypothesis of Proposition \ref{prop:q1} holds true. Recall that for a controlled process $z$, the generalized time derivative $\dot{z}$ is introduced in Definition~\ref{def:time_deriv} . Then the $q$-function in (\ref{eq:q'}) admits the expression
     \begin{equation}\label{eq:q''}
         q(t,y,\mu;\ga)=\dot{J}_{t,T}(y,\ga)+H(t,y,\mu,\nabla_{x}J_{t,T}(y,\ga))-{\nabla_{x}J_{t,T}(y,\ga)\cdot \int_{U}b(t,y,u)\ga_{t}(du)}.
     \end{equation}
     \begin{proof}
     Viewing $J_{t,T}(y,\ga)$ as a function of $t$, one can compute its increments to show that 
\begin{equation}
   \label{eq:rf-d eriv} \dot{J}_{t,T}(y,\ga)=-F(t,y,\ga_{t}),
\end{equation}
where $\dot{J}_{t,T}(y,\ga)$ is the notation explained in (\ref{eq:time-not}). 
        Therefore $(\ref{eq:q''})$ is a direct consequence of (\ref{eq:rf-d eriv}) and~(\ref{eq:q'}).
\end{proof}
\end{proposition}

\section{Gibbs measure representation for optimal policy}~\label{sec:gibbs}

This section further explores cases of reward
functions explicitly containing entropy terms.
This will allow to translate the notion of a $q$-function developed in Section~\ref{sec:q-function} into a more computational optimization tool. To this aim, assume the instant reward function $F$ in equation~\eqref{eq:cont-1} takes the form {$(\ref{eq:F-ex})$}, recalled here for the reader's convenience:
\begin{equation}\label{eq: ConvF}
    F(s,x,\ga)=\int_{U}R(s,x,u)\dot{\ga}_{s}(u)\,du-\lambda\int_{U}\dot{\ga}_{s}(u)\log \dot{\ga}_{s}(u)\,du.
\end{equation}
This type of reward is one of the most common one in the context of reinforcement learning. For the sake of simplicity, we assume the discount factor $\rho = 0$. It is important to note that this assumption can be readily relaxed. We will examine the open loop case in Section~\ref{sec:open-loop} and the closed loop case in Section~\ref{sec:closed-loop}.

\subsection{Open loop control and $q$-function}\label{sec:open-loop}
In this section we derive an expression for the open loop control (that is an optimal policy derived once for all) in case of a reward with entropy term. We first define 
\begin{equation}\label{g1}
\tilde{H}(t, y,a,p) := p \cdot b(t,y,a) + R(t,y,a),
\end{equation}
that is $\tilde{H}$ is the function integrated to produce the Hamiltonian in \eqref{eq:Hamil}, without the entropy part. This also means that the rough HJ equation~\eqref{d3} now takes the form
\beq\label{eq:hjb-rough}
	\partial_t v(t,y) + \sup_{\gamma \in \cp(U)} \int_U \left[\tilde{H}(t,y,a,\nabla v(t,y))- \lambda \log \dot \gamma(a) \right] \dot\gamma(a) da
	 + \nabla v(t,y) \cdot \sigma(t,y) \dot \zeta_t = 0
\eeq
with terminal condition $v(T,x) = G(x)$. We can also recast \eqref{g1}-\eqref{eq:hjb-rough} by saying that in this section, we consider a Hamiltonian of the form 
\beq\label{eq:H}
H(t,y,\ga,p) = \int \lc \tilde{H}(t,y,a, p) - \lambda \log \dot{\ga}(a) \rc \dot{\ga}(a) da,
\eeq
with $\tilde{H}(t,y,a,p)$ given by \eqref{g1}.
\begin{remark} \label{rmk: K}
    As mentioned in Remark \ref{rmk: F}, a function $F$ like in $(\ref{eq: ConvF})$ does not exactly fit into our general framework. Specifically, Hypothesis \ref{hyp:reward} is not satisfied by a function $F$ including an entropy term like in $(\ref{eq: ConvF})$. However, one can circumvent this technical problem by replacing the state space $\mathcal{K}\subset \mathcal{P}(U)$ with a set
    \begin{equation}
        \mathcal{K}_{\beta}=\left\{\mu \in \mathcal{P}(U): \mu \text{ admits a density }\dot{\mu} \text{ with } \dot{\mu}(u)\geq\beta \right\},
    \end{equation}
    for a constant $\beta>0$, and assuming $U$ is a compact domain of $\R^d$. We let the reader check that all our considerations on the HJ equation (see Theorem \ref{thm:HJBequation}) go through on the space $\mathcal{K}_{\beta}$. Moreover, on $\mathcal{K}_{\beta}$ the function $F$ in $(\ref{eq: ConvF})$ satisfies Hypothesis \ref{hyp:reward}.  
\end{remark}
In view of~\eqref{eq:hjb-rough}, it is natural to formulate the following conjecture:
whenever one starts a dynamics like \eqref{d1} from an initial condition $x$ at time $t$, the optimal policy $\ga^{\ast}$ is such that $\ga_t^{\ast}$ is a Gibbs measure which can be written as 
\beq\label{eq:ha*-feedback}
	\dot{\gamma}^*(\cdot|\, t,x) \propto \exp \left[\dfrac{1}{\lambda} \tilde H(t,x,\cdot,\nabla v(t,x) ) \right].
\eeq
In order to prove \eqref{eq:ha*-feedback}, we first recall a general optimization result in presence of an entropy term. Although it is inspired by \cite[Lemma~2]{jia2023qlearning}, we include a simpler proof here for sake of clarity.
 \begin{lemma}\label{lem:opt-gibbs}
 Let $\la > 0$ and a measurable function $h: U \to \R$ satisfying the assumption $\int_{U} \exp(\frac{1}{\la}h(a)) \, da < \infty$, where $U$ is the domain introduced in Definition~\ref{def:rel-control}. Then 
 \beq\label{eq:soft-optimizer} 
 \dot{\ga}^{\ast}(a) = \dfrac{\exp(\frac{1}{\la}h(a))}{\int_U \exp(\frac{1}{\la}h(a)) da}
 \eeq 
 is the unique maximizer of the following problem
 \beq\label{eq:soft-opt-obj}
 L^{\ast}(h) \equiv \sup_{\ga \in \cp(U)} L(\ga, h),
 \eeq
 with
 \beq\label{eq:L}
 L(\ga, h) = \int_U \lc h(a) - \la \log \dot\ga (a) \rc \dot \ga(a) da.
 \eeq
 Further the optimal value achieved in \eqref{eq:soft-opt-obj} is explicitly given by
 \begin{equation}\label{g2}
 L^{\ast}(h) = \la \log \lc \int_U \exp \lp \frac{h(a)}{\la} \rp da \rc.
 \end{equation}
 \end{lemma}
 \begin{proof}
This is essentially Donsker-Varadhan's variational formula \cite{donsker-varadhan1983}, but a proof is included for completeness. 
Let $Z = \int_U \exp(\frac{1}{\la}h(a)) da$ so that $\dot\ga^{\ast}(a) = \exp(\frac{1}{\la} h(a))/Z$. Define for $\ga \in \cp(U)$ admitting densities, the objective
\beq\label{eq:soft-obj}
\Pi(\ga)= L(\ga,h) = \int_U \lc h(a) - \la \log \dot\ga (a) \rc \dot \ga(a) da.
\eeq
Then starting from \eqref{eq:soft-obj}, some elementary algebraic manipulations show that
$$
\Pi(\ga) = \la \int_U \lc \log(\exp(h(a)/\la)) - \log \dot\ga (a) \rc \dot \ga(a) da.
$$
Next recalling that $\exp(h(a)/\la) = Z \dot\ga^{\ast}(a)$ from \eqref{eq:soft-optimizer}, we easily get
$$
\Pi(\ga) = \la \log Z + \la \int_U \log \lp \dfrac{\dot{\ga}^{\ast}(a)}{\dot\ga (a)} \rp \dot \ga(a) da
 = \la \log Z - \la \textnormal{KL}(\ga \| \ga^{\ast}),
$$
where we have involved the standard definition of KL distance for the second identity.
Since $\textnormal{KL}(\ga \| \ga^{\ast}) \geq 0$ with equality iff $\ga = \ga^{\ast}$ it follows that
$$
\sup_{\ga \in \cp(U)} \Pi(\ga) = \la \log Z
= \la \log \lc \int_U \exp \lp \frac{h(a)}{\la} \rp da \rc,
$$
and the unique maximizer is $\ga^{\ast}$. This shows~\eqref{g2} and finishes our proof.
\end{proof}

\noindent
We now apply Lemma~\ref{lem:opt-gibbs} to the context of our rough HJ equation, solving the conjecture~\eqref{eq:ha*-feedback} mentioned above.
\begin{proposition}\label{prop:feedback-policy}
Assume that the coefficients $b$ and $\si$ satisfy Hypothesis~\ref{hyp:coeff}. The path $\bzeta$ is such that Hypothesis~\ref{hyp:zeta} is fulfilled, and the reward $F$ verifies Hypothesis~\ref{hyp:reward}.
Let then $v$ be the unique solution of equation~\eqref{eq:hjb-rough}. 
For notational sake, let us also use the following abbreviation: the function $\tilde{H}:U\to\R$ is defined by
\begin{equation}\label{g3}
\tilde{H}(\cdot) := \tilde{H}(t,y,\cdot,\nabla v(t,y)) .
\end{equation}
Then  the feedback-policy
\beq\label{eq:ga*}
\dot\ga_t^{\ast}(a \mid y, \nabla v(t,y)) \propto \exp \lp \dfrac{1}{\la}\tilde{H}(t,y,a, \nabla v(t,y)) \rp
\eeq
is such that
\beq\label{eq:L*-2}
L^{\ast}(\tilde{H}) \equiv 
\sup_{\ga \in \cp(U)} L(\ga, \tilde{H})
=
L({\ga}_t^{\ast}, \tilde{H}) ,
\eeq
where $L$ is given by \eqref{eq:L}.
Furthermore, still with notation~\eqref{g3} in mind, the optimal value in \eqref{eq:L*-2} is explicitly given by
\beq\label{eq:L*-2-value}
L^{\ast}(\tilde{H}) = \la \log \lc \int_U \exp\lp \frac{1}{\la} \tilde{H} (t,y,a, \nabla v(t,y))\rp da \rc.
\eeq
\end{proposition}
\begin{proof}
This is a direct application of Lemma~\ref{lem:opt-gibbs}.
\end{proof}
\begin{remark}\label{rem:opt-policy}
Thanks to Proposition~\ref{prop:feedback-policy}, the rough-HJB in \eqref{eq:hjb-rough} can be re-written as
\beq\label{eq:hjb-rough-2}
\partial_t v(t,y) +\la \log \lc \int_U \exp\lp \frac{1}{\la} \tilde{H} (t,y,a,\nabla v(t,y))\rp da \rc
	 + \nabla v(t,y) \cdot \sigma(t,y) \dot \zeta_t = 0
\eeq
Further notice that for each given initial state $x_0$, the feedback policy in turn generates an open-loop control dictated by~\eqref{eq:ga*}:
$$
\dot{\ga}_t^{\ast}(a)  = \dot{\ga}_t^{\ast}(a|x_t^{\ga^{\ast}}, \nabla v(t, x_t^{\ga^{\ast}})) \propto \exp \lp \frac{1}{\la} \tilde{H}(t, x_t^{\ga^{\ast}}, a, \nabla v(t, x_t^{\ga^{\ast}})) \rp,
$$
where $v$ solves the rough-HJB \eqref{eq:hjb-rough-2}, and $x_t^{\ga^{\ast}}$ solves our rough differential equation~\eqref{d1} with initial state $x_0$ and policy $\ga^{\ast}$.
Since $J_{t,T}(y, \ga^{\ast})$ \emph{solves}~\eqref{eq:hjb-rough} or \eqref{eq:hjb-rough-2} we have that the optimal open-loop control that solves the optimization problem \eqref{eq:V_sy} satisfies
\beq\label{eq:opt-policy-2}
\dot{\ga}_t^{\ast}(a)  = \dfrac{\exp \lp \frac{1}{\la} \tilde{H}(t, x_t^{\ga^{\ast}}, a, \nabla_x J_{t,T}(x_t^{\ga^{\ast}}; \ga^{\ast})) \rp}{\int_U \exp \lp \frac{1}{\la} \tilde{H}(t, x_t^{\ga^{\ast}}, a, \nabla_x J_{t,T}(x_t^{\ga^{\ast}}; \ga^{\ast})) \rp da}.
\eeq
Otherwise stated, the $\sup$ contained in the right hand side of equation~\eqref{eq:hjb-rough} is not even needed to compute ${\ga}_t^{\ast}$.
\end{remark}

With Proposition~\ref{prop:feedback-policy} and 
Remark~\ref{rem:opt-policy} in hand, we can now make a crucial link between our optimization procedure and  the $q$-function defined in \eqref{eq:litq1}.
\begin{proposition} ~\label{prop:closed-gibbs-q}
Let $\ga^{\ast}$ be the optimal open-loop policy defined by \eqref{eq:ga*}. Then we can also express $\ga^{\ast}$ as
%The optimal open-loop policy $\ga^{\ast}$ satisfies
\beq\label{eq:ga*-q}
\dot\ga_t^{\ast}(a) = \exp \lp \dfrac{1}{\la} q(t,y,a; \ga^{\ast}) \rp,
\eeq
where the $q$-function is defined in \eqref{eq:litq1}.
\end{proposition}
\begin{proof}
Recall our abbreviation~\eqref{g3} for $\tilde{H}$.
According to \eqref{eq:L*-2}, we have
\beq\label{eq:e}
L^{\ast} (\tilde H) = \sup_{\ga \in \cp(U)} L(\ga, \tilde H) = L(\ga_t^{\ast}, \tilde H).
\eeq
Moreover, recall from \eqref{eq:L} that
$$
L(\ga_t^{\ast}, \tilde H) 
= \int_U \lc \tilde H(t,y,a,\nabla v(t,y)) - \lambda \log \dot{\ga}_t^{\ast}(a) \rc \dot{\ga}_t^{\ast}(a) da. 
$$
Comparing the above expression with \eqref{eq:H}, we get
$$
L(\ga_t^{\ast}, \tilde H) = H(t, y, \ga_t^{\ast}, \nabla v(t,y)) .
$$
Recalling from \eqref{eq:e} that $L(\ga_t^{\ast}, \tilde H) = L^{\ast}(\tilde H)$ and applying \eqref{eq:L*-2-value} plus the fact that $\nabla v (t,x_{t}^{\ga^{*}})=\nabla J_{t,T}(y,\ga^{*})$, we end up with the relation 
%Notice that for the optimal policy $\ga^{\ast}$, thanks to %Proposition~\ref{prop:feedback-policy}  we have
\beq\label{eq:H-opt}
H(t,y, \ga_t^{\ast}, \nabla_x J_{t,T}(y; \ga^{\ast})) = \la \log \lp \int_U \exp\lp \frac{1}{\la} \tilde{H}(t,y,a,\nabla_x J_{t,T}(y;\ga^{\ast})) \rp da \rp.
\eeq
On the other hand, one can apply \eqref{eq:H} for a generic point $a \in U$ and $\ga = \der_a$. In this case, the usual convention is to set the entropy term at 0 (see e.g \cite{jia2023qlearning}). Hence it is readily checked that 
\begin{eqnarray}\label{H-opt-a}
H(t,y,\der_a, \nabla_x J_{t,T}(y ; \ga^{\ast})) 
&=& 
\nabla_x J_{t,T}(y;\ga) \cdot b(t,y,a) + R(t,y,a) \notag\\
&=& \tilde{H}(t,y,a,\nabla_x J_{t,T}(y ; \ga^{\ast})) .
\end{eqnarray}
Thus from \eqref{eq:q'} the optimal $q$-function satisfies
\begin{align}\label{eq:f}
&q(t,y,a;\ga^{\ast}) = H(t,y,\der_a, \nabla_x J_{t,T}(y ; \ga^{\ast})) - H(t, y, \ga_t^{*}, \nabla_x J_{t,T}(y;\ga^{\ast})) \nonumber \\
&= \tilde{H}(t,y,a,\nabla_x J_{t,T}(y; \ga^{\ast})) - \la \log \lp \int_U \exp\lp \frac{1}{\la} \tilde{H}(t,y,a,\nabla_x J_{t,T}(y;\ga^{\ast})) \rp da \rp,
\end{align}
where we have invoked \eqref{eq:H-opt} and \eqref{H-opt-a} for the second identity.
Furthermore, thanks to elementary algebraic manipulations, one can recast \eqref{eq:f} as 
$$
\exp \lp \dfrac{1}{\lambda} q(t,y,a;\ga^{\ast}) \rp 
=
\dfrac{\exp\lp \frac{1}{\la} \tilde H (t, y, a, \nabla J_{t,T}(y; \ga^{\ast})) \rp}{\int_U \exp\lp \frac{1}{\la} \tilde H (t, y, a, \nabla J_{t,T}(y; \ga^{\ast})) \rp da} .
$$
Comparing this expression with \eqref{eq:opt-policy-2}, we trivially obtain that \eqref{eq:ga*-q} holds true. This finishes our proof.
\end{proof}

\subsection{Policy Improvement}\label{sec:closed-loop}

In this section, we will assume that our policies are closed loop instead of open loop. This case corresponds to a situation where one has the ability to improve the current policy $\ga$ according to the data received. We will see how this can be achieved in the entropy reward case.

The closed loop strategy begins by considering the objective functional $J_{t,T} $ in \eqref{eq:J} related to a given strategy $\ga$. Similarly to Theorem~\ref{thm:HJBequation} and Theorem~4.22 in \cite{CHT}, one can prove that $J_{t,T} \equiv v^{\ga}$ solves the following rough PDE:
\beq\label{eq:pde-ga}
\partial_t v^{\ga}(t,y) + H(t, y, \ga, \nabla v^{\ga}(t,y)) + \nabla v^{\ga}(t,y) \cdot \si(t,y) \dot{\zeta}_t = 0 \, ,
\eeq
where $H$ is still the Hamiltonian given by~\eqref{eq:H}.
We are now facing the following issue: while policy improvement methods are based on monotonicity for PDEs, equation~\eqref{eq:pde-ga} lacks this type of monotonicity due to the noisy term. In order to solve that problem we resort to stochastic characteristics. This method has already been employed in \cite{CHT}, which we will refer to for details of proof (see also~\cite{CFO} for similar considerations). We start by recalling the basic definition of the rough flow used in our setting.

\begin{proposition}\label{prop:rough-flow-a}
Let $\bzeta$ be a rough path satisfying Hypothesis~\ref{hyp:zeta}. We also consider a coefficient $\si$ verifying Hypothesis~\ref{hyp:coeff}, modified to $\sigma \in \mathcal C_b^{1,4}([0,T]\times \R^m; \R^{m,d})$. Then the following holds:
\begin{enumerate}[wide, labelwidth=!, labelindent=0pt, label=\emph{(\roman*)}]
\setlength\itemsep{.1in}  
    \item For any initial condition $\eta \in \mathbb{R}^m$, there exists a unique solution to the rough equation
    \beq\label{eq:rough-flow-b}
    \phi_t(\eta) = \eta + \int_0^t \si(r, \phi_r(\eta))\,d\bzeta_r .
    \eeq
    \item The function $\eta \mapsto \phi_t(\eta)$ is a diffeomorphism of $\mathbb{R}^m$.
    \item The inverse $\chi_t(\eta) = \phi_t^{-1}(\eta)$ solves the rough equation
    $$
    \chi_t(\eta) = \eta - \int_0^t \nabla \chi_r(\eta) \cdot \si(r, \eta) d \zeta_r.
    $$
\end{enumerate}
\end{proposition}

One can now transform the rough PDE~\eqref{eq:pde-ga} into a deterministic PDE with noisy coefficients, using compositions with the flow. This is summarized in the proposition below, which can be proved along the same lines as \cite[Proposition~4.21]{CHT} invoking the composition rule in Proposition \ref{prop:str_comp}. We omit this proof here for the sake of conciseness.

\begin{proposition}\label{prop:rough-flow-c}
Under the conditions of Proposition~\ref{prop:rough-flow-a}, let us consider the path $\hat{v}^{\ga} = v^{\ga} \circ \phi$, where $v^{\ga}$ solves \eqref{eq:pde-ga} and $\phi$ is given by \eqref{eq:rough-flow-b}. Then $v^{\ga}$ is a solution to the following PDE with noisy coefficients:
$$
\partial_t \hat{v}_t^{\ga}(y) + \hat{H}(t,y,\ga, \nabla \hat{v}_t^{\ga}(y)) = 0,
$$
where the Hamiltonian $\hat{H}$ is given by
\beq\label{eq:H-hat-b}
\hat{H}(u, y, \ga, p) = H(u, \phi_u(y), \ga, p \cdot \nabla \phi_u^{-1}(\phi_u(y))).
\eeq
%and where $\phi$ solves
%$$
%\phi_t(\eta) = \eta + \int_0^t \si^k (r, \phi_r(\eta)) d \zeta_r^k.
%$$
\end{proposition}

Thanks to Proposition~\ref{prop:rough-flow-c}, we are now reduced to the study of an ordinary PDE interpreted in the viscosity sense. For those equations we have the following comparison principle.

%\begin{lemma}\label{lem:comparison}
%Let $H_1, H_2 : [0,T] \times \R^m \times \R^m \to \R$ be continuous satisfying
%$$
%H_1(t, y, p) \leq H_2(t, y, p),
%$$
%for all $(t,y,p)$. Let $v_1$ be a viscosity subsolution of
%$$
%\partial_t v_1(t,y) + H_1(t,y,\nabla v(t,y)) = 0,
%$$
%and $v_2$ be a viscosity supersolution of 
%$$
%\partial_t v_2(t,y) + H_2(t, y, \nabla v(t,y)) = 0.
%$$
%Then if $v_1(T,y) = v_2(T,y)$ we have
%$$
%v_1(t,y) \leq v_2(t, y) \text{ for all }(t,y) \in [0,T]\times \R^n.
%$$
%\end{lemma}

\begin{lemma}\label{lem:comparison}
Let $H_1, H_2 : [0,T] \times \R^m \times \R^m \to \R$ be continuous functions satisfying
\[
H_1(t,y,p) \leq H_2(t,y,p), 
\quad \text{for all}\quad (t,y,p) \in [0,T]\times\R^m\times\R^m,
\]
and such that for all $(t,y)\in[0,T]\times\R^m$ and all $p_1,p_2\in\R^m$,
\[
|H_i(t,y,p_1) - H_i(t,y,p_2)| \le L\,|p_1 - p_2|, \qquad i=1,2,
\]
for some constant $L>0$.
Let $\hat{v}_1$ be a bounded, upper semicontinuous viscosity subsolution of
\[
\partial_t \hat{v}_1(t,y) + H_1(t,y,\nabla \hat{v}_1(t,y)) = 0,
\]
and let $\hat{v}_2$ be a bounded, lower semicontinuous viscosity supersolution of
\[
\partial_t \hat{v}_2(t,y) + H_2(t,y,\nabla \hat{v}_2(t,y)) = 0,
\]
such that $v_1(T,y) \le v_2(T,y)$ for all $y\in\R^m$. Then it holds 
\[
\hat{v}_1(t,y) \le \hat{v}_2(t,y), 
\quad \text{for all}\quad (t,y)\in[0,T]\times\R^m.
\]
\end{lemma}

\begin{proof}
This result is a direct application of the variation of \cite[Theorem~8.2]{crandall-ishi-lyons} spelled out on \cite[p. 52]{crandall-ishi-lyons}. In that context, just take $F \geq G$ ($H_2 \geq H_1$ in our notation) and $g = 0$. We also assume $v_1(T, \cdot) \leq v_2(T, \cdot)$, which is $\hat{u} \leq \hat{v}$ in \cite{crandall-ishi-lyons}. The inequality displayed on p. 52 then reads $\hat{v}_1 \leq \hat{v}_2$, which is the desired conclusion.
\end{proof}

Let us now state the policy improvement theorem which is the main result of this section. 

\begin{theorem}%[To be checked]
We work under the same conditions as in Proposition~\ref{prop:feedback-policy}. For a generic policy $\ga_t \in \cp(U)$, define the value $v_t^{\ga}(y) = J_{t,T}(y;\ga)$ as in \eqref{eq:J}. Next define a new policy $\ga^{+}$ by the formula 
\beq\label{eq:ga+}
\dot{\ga}^{+} (a \mid y) = \exp\lp \frac{1}{\la} q(t,y,a;\ga) \rp \, ,
\eeq
where $q$ is the function defined by~\eqref{eq:litq1}.
Then the value $v_t^{+}(y) = J_{t,T}(y, \ga^{+})$ satisfies
\beq\label{eq:value-improved}
v_t^{+}(y) \geq v_t^{\ga}(y), \quad \text{for all }(t,y) \in [0,T] \times \R^d. 
\eeq
\end{theorem}

\begin{proof}
Recall that $v_t^{\ga}(y) = J_{t,T}(y;\ga)$ solves equation~\eqref{eq:pde-ga}. Therefore, along the same lines as for Proposition~\ref{prop:feedback-policy}, we have that $\ga^{+}$ defined by \eqref{eq:ga+} is such that 
$$
\sup_{\eta \in \cp(U)} L(\eta, \tilde{H}) = L(\ga_t^{+}, \tilde{H}),
$$
where $\tilde{H}$ is defined by \eqref{g1}-\eqref{g3} and $L$ is given by \eqref{eq:L}. We now set
$$
H_{1}=L(\ga, \tilde{H}), \quad H_{2}=L(\ga_t^{+}, \tilde{H}),
$$
and also define $\hat{H}_1$, $\hat{H}_2$ from $H_1$, $H_2$ as in \eqref{eq:H-hat-b}. Then we have $H_2 \geq H_1$, thus $\hat{H}_2 \geq \hat{H}_1$, and Lemma~\ref{lem:comparison} asserts that
$$
\hat{v}_t^{+}(y) \geq \hat{v}_t^{\ga}(y).
$$
Composing with $\phi^{-1}$ from Proposition~\ref{prop:rough-flow-a} yields our desired conclusion~\eqref{eq:value-improved}.
\end{proof}

\section*{Acknowledgements}
\noindent
All four authors are partially supported for this work by NSF grant  DMS-2153915. H. Honnappa is also partly supported through NSF grant CMMI-22014426. ST would like to thank François Delarue for some valuable conversation about the Gibbs measure setting for relaxed controls.

%\bigskip

%\bibliographystyle{plain}
\bibliographystyle{abbrvnat}
\bibliography{problem-2.bib}

\end{document}